\documentclass{siamltex}
\usepackage{amssymb,amsmath,amsfonts,amscd}
\usepackage{algorithm,algorithmic}
\usepackage{upgreek,color,url}
\usepackage{epsfig}
\usepackage{graphicx}
\usepackage{comment}
\usepackage{stylefile}   
\title{Probabilistic Error Analysis for  Inner Products\thanks{Funding: The work of the first author was supported in part by 
National Science Foundation grants DMS-1745654 and DMS-1760374. The work of the second author was supported in part
by National Human Genome Research Institute grant HG006139, and the National Institute of General Medical Sciences grant GM053275.}}
\author{Ilse C.F. Ipsen\thanks{Department of Mathematics, North Carolina State University, Raleigh, NC 27695-8205, USA,
ipsen@ncsu.edu} \and Hua Zhou\thanks{Department of Biostatistics, University of California, Los Angeles, CA 90095-1772, USA, huazhou@ucla.edu}}

\begin{document}
\maketitle

\begin{abstract}
Probabilistic models are proposed for bounding the forward error in the numerically computed
inner product (dot product, scalar product) between of two real $n$-vectors.
We derive probabilistic perturbation bounds, as well as probabilistic roundoff error bounds
for the sequential accumulation of the inner product.
These bounds are non-asymptotic, explicit, and make minimal assumptions on perturbations and roundoffs.

The perturbations are represented as  independent, bounded, zero-mean random variables,
and the probabilistic perturbation bound is based on Azuma's inequality.
The roundoffs are also represented as bounded, zero-mean random variables.
The first probabilistic bound assumes that the roundoffs are independent, 
while the second one does not. For the latter, we construct a Martingale
that mirrors the sequential order of computations.

Numerical experiments confirm that our bounds are more informative,
often by several orders of magnitude, than traditional deterministic 
bounds -- even for small vector dimensions~$n$ and very stringent success probabilities.
In particular the probabilistic roundoff error bounds are functions of $\sqrt{n}$ rather than~$n$,
thus giving a quantitative confirmation of Wilkinson's intuition. The paper concludes with a critical assessment of 
the probabilistic approach.
\end{abstract}


\begin{AM}
65F30, 65G50, 60G42, 60G50
\end{AM}

\section{Introduction}
Probabilistic approaches towards roundoff analysis have been applied to: matrix inversion
by von Neumann \& Goldstine \cite{vNG47} and Tienari \cite{Tie70};
matrix addition and multiplication, and Runge Kutta methods by Hull \& Swenson \cite{HS66}; 
solution of ordinary differential equations by Henrici \cite{Hen63}; 
Gaussian elimination by Barlow \& Bareiss  \cite{BB80,BB85b,BB85a}; 
convolution and FFT by Calvetti \cite{Cal91a,Cal91b,Cal92}; 
solution of eigenvalue problems by Chatelin \&  Brunet \cite{BBC88,BC86,CB90}; 
LU decomposition and linear system solution by Babu\v{s}ka \& S\"{o}derlind \cite{BS18}
 and Higham and Mary \cite{HM18}.
 Yet, the futility of probabilistic roundoff error analysis has also been 
pointed out  \cite[page 2]{HS66}, \cite[Page 17]{Kah96}, since roundoffs apparently do not behave
 like random variables. 

Nevertheless, we present 
 probabilistic perturbation and roundoff error bounds for the forward error
in the numerically computed inner product\footnote{The superscript~$T$ denotes the transpose,
and for relative bounds we assume $\vx^T\vy\neq 0$.},
\begin{eqnarray*}
\vx^T\vy = x_1 y_1 +\cdots + x_ny_n,
\end{eqnarray*}
between  two real $n$-vectors
\begin{eqnarray*}
\vx=\begin{pmatrix}x_1 \\\vdots \\ x_n\end{pmatrix}\in\rn \qquad \mathrm{and}\qquad 
\vy=\begin{pmatrix}y_1\\ \vdots \\ y_n\end{pmatrix}\in\rn.
\end{eqnarray*}

\paragraph{Contributions}
The idea is to represent perturbations and roundoffs  as random variables, express the total forward error
as a sum of "local" forward errors, and then apply a concentration inequality to the sum.
In contrast to some of the previous work, the roundoffs are not required to obey a particular probability distribution.
We "motivate" the particular form of each probabilistic bound with a corresponding deterministic bound, and interpret
the various random variables in terms of particular forward errors.

Our probabilistic approach is most closely related to that of Higham and Mary~\cite{HM18} 
who derive backward error bounds. In contrast, our forward error bounds 
lead to new condition numbers (Sections \ref{s_perturb} and~\ref{s_prob1}), and they are tighter
because they avoid a union bound for the probabilities.
 Our bounds
are also simple, intuitive, and easy to interpret, with a clear relationship between failure probability and relative
error. Compared to \cite[Theorem 3.1]{HM18}, our Corollary~\ref{c_4} is tighter  and does not assume independence
of roundoffs.

\paragraph{Overview}
To facilitate the introduction of the probabilistic approach, we start as simple as possible,
with probabilistic perturbation bounds (Section~\ref{s_perturb}). The
perturbations are represented as independent, bounded, zero-mean random variables; and 
the forward error is bounded by Azuma's inequaility.
This is followed by probabilistic  roundoff error bounds for the sequential accumulation of inner products (Section~\ref{s_roundoff1}).
The roundoffs are represented as independent, bounded, zero-mean random variables;
and the forward error is, again, bounded by Azuma's inequaility.
However, numerical experiments (Section~\ref{s_numexp}) illustrate that for non-negative vectors of large dimension,
the probabilistic expression stops being an upper bound. By way of an explanation,
Henrici ends his 1963 paper \cite[page 11]{Hen63} with:

\begin{quote}
The crucial hypothesis for the above statistical theories is the hypothesis
of independence of local errors.
While this assumption seems to yield realistic results in many cases, some situations are known, [...],
where local errors definitely cannot be considered to be independent.
To elucidate the conditions under which local errors act like independent
variables would seem to be a fascinating if difficult problem.
\end{quote}

As a consequence, and in contrast to \cite{HM18}, we relinquish the independence assumption 
and derive a general probabilistic roundoff error bound (Section~\ref{s_roundoff2}).
The roundoffs are represented as bounded, zero-mean random variables; and the forward error
is bounded by an Azuma-Hoeffding Martingale.
In particular, we present a quantitative confirmation of Wilkinson's intuition \cite[Section 1.33]{Wilk94book}
that the roundoff error in $n$ operations is proportional to $\sqrt{n}\>u$ rather than $n\>u$.
The paper ends with a critical analysis of the probabilistic approach, and a long list of future work (Section~\ref{s_conc}).

\section{Perturbation bounds}\label{s_perturb}
To calibrate the roundoff error bounds and set the stage for the probabilistic approach, we start off with perturbation bounds:
first, deterministic bounds that 
generalize the traditional bound and motivate the 
probabilistic bound (Section~\ref{s_perturb1}),
and then the probabilistic bound (Section~\ref{s_perturb2}).

We use the Hadamard product 
\begin{eqnarray*} 
\vx\circ\vy \equiv \begin{pmatrix} x_1y_1 & \cdots & x_ny_n\end{pmatrix}^T
\end{eqnarray*}
to compactly express componentwise relative perturbations as
\begin{eqnarray*}
\vhx=\begin{pmatrix}(1+\delta_1)\, x_1\\ \vdots \\ (1+\delta_n)\, x_n\end{pmatrix}
=\vx+\vdelta\circ \vx,\qquad
\vhy=\begin{pmatrix}(1+\theta_1)\,y_1\\ \vdots \\ (1+\theta_n)\,y_n\end{pmatrix}
 = \vy+\vtheta\circ\vy,
 \end{eqnarray*}
where $|\delta_k|,|\theta_k|\leq u$, $1\leq k\leq n$, for some $u>0$, and the perturbation vectors are
\begin{eqnarray*}
\vdelta \equiv \begin{pmatrix} \delta_1 & \cdots &\delta_n\end{pmatrix}^T,\qquad
\vtheta\equiv \begin{pmatrix} \theta_1 & \cdots & \theta_n\end{pmatrix}^T.
\end{eqnarray*}

\subsection{Deterministic perturbation bound}\label{s_perturb1}
We generalize the traditional perturbation bound to a whole class of bounds, and single
out a specific bound to motivate the probabilistic bound in Section~\ref{s_perturb2}.

\begin{theorem}\label{t_2}
If  $\frac{1}{p}+\tfrac{1}{q}=1$, then the relative forward error in the perturbed inner product is bounded by
\begin{eqnarray*}
\left|\frac{\vhx^T\vhy-\vx^T\vy}{\vx^T\vy}\right| \leq 
\frac{\|\vx\circ\vy\|_p}{|\vx^T\vy|}\>\|\vdelta+\vtheta+\vdelta\circ\vtheta\|_q.
\end{eqnarray*}
\end{theorem}

\begin{proof}
From associativity, distributivity and the fact that all quantities are real follows
\begin{eqnarray*}
\vhx^T\vhy-\vx^T\vy&=&(\vdelta\circ \vx)^T\vy+\vx^T(\vtheta\circ \vy)+(\vdelta\circ\vx)^T(\vtheta\circ\vy)\\
&=& \sum_{k=1}^n{x_ky_k\>\left(\delta_k+\theta_k+\delta_k\theta_k\right)}
= (\vx\circ\vy)^T\left(\vdelta+\vtheta+\vdelta\circ\vtheta\right).
\end{eqnarray*}
The H\"{o}lder inequality implies
\begin{eqnarray*}
\left|(\vx\circ\vy)^T\left(\vdelta+\vtheta+\vdelta\circ\vtheta\right)\right|
\leq  \|\vx\circ\vy\|_p\>\|\vdelta+\vtheta+\vdelta\circ\vtheta\|_q.
\end{eqnarray*}
$\quad$
\end{proof}

Below is a specialization of Theorem~\ref{t_2} to popular $p$-norms.

\begin{corollary}\label{c_2}
Theorem~\ref{t_2} implies the following bounds.
\begin{enumerate}
\item Traditional bound $(p=1)$
\begin{eqnarray*}
\left|\frac{\vhx^T\vhy-\vx^T\vy}{\vx^T\vy}\right| \leq 
\frac{\|\vx\circ\vy\|_1}{|\vx^T\vy|}\>\|\vdelta+\vtheta+\vdelta\circ\vtheta\|_{\infty}\leq
\frac{|\vx|^T|\vy|}{|\vx^T\vy|}\>u(2+u).
\end{eqnarray*}
\item Same amplifier as in Theorem~\ref{t_1} $(p=2)$
\begin{eqnarray}\label{e_pb2}
\left|\frac{\vhx^T\vhy-\vx^T\vy}{\vx^T\vy}\right| &\leq& 
\frac{\|\vx\circ\vy\|_2}{|\vx^T\vy|}\>\|\delta+\vtheta+\vdelta\circ\vtheta\|_2\\
&\leq & \sqrt{n}\>\frac{\|\vx\circ\vy\|_2}{|\vx^T\vy|}\>u(2+u).\nonumber
\end{eqnarray}
\item Smallest amplifier $(p=\infty)$
\begin{eqnarray*}
\left|\frac{\vhx^T\vhy-\vx^T\vy}{\vx^T\vy}\right| \leq 
\frac{\|\vx\circ\vy\|_{\infty}}{|\vx^T\vy|}\>\|\vdelta+\vtheta+\vdelta\circ\vtheta\|_1\leq
n\>\frac{\|\vx\circ\vy\|_{\infty}}{|\vx^T\vy|}\> u(2+u).
\end{eqnarray*}
\end{enumerate}
\end{corollary}

\begin{proof}
The traditional bound follows from 
\begin{eqnarray*}
\|\vx\circ\vy\|_1=\sum_{k=1}^n{|x_ky_k|}=\sum_{k=1}^n{|x_k|\,|y_k|}=|\vx|^T|\vy|.
\end{eqnarray*}
$\quad$
\end{proof}

The numerical experiments in Section~\ref{s_perturbamp} 
suggest that the three bounds tend to differ by at most an order of magnitude or so, 
with the traditional bound  being the tightest.

\subsection{Probabilistic perturbation bound}\label{s_perturb2}
We derive a probabilistic bound corresponding to the deterministic bound (\ref{e_pb2}), and then compare the two bounds.

The basis for the probabilistic bounds is a concentration inequality, which bounds the deviation of a sum from its mean
in terms of the deviations of the individual summands from their means.

\begin{lemma}[Azuma's inequality, Theorem 5.3 in \cite{ChungLu2006}]\label{l_1}
Let $Z\equiv Z_1+\cdots + Z_n$ be a sum of independent random variables $Z_1, \ldots, Z_n$ with 
\begin{eqnarray*}
|Z_k-\E[Z_k]|\leq c_k, \qquad 1\leq k\leq n.
\end{eqnarray*}
Then for any $0<\delta<1$, with probability at least $1-\delta$,
\begin{eqnarray*}
\left|Z - \E[Z]\right|\leq \sqrt{\sum_{k=1}^n{c_k^2}}\>\sqrt{2\ln{(2/\delta)}}.
\end{eqnarray*}
\end{lemma} 

\begin{proof}
In \cite[Theorem 5.3]{ChungLu2006} set
\begin{eqnarray*}
\delta\equiv \Pr\left[ |Z-\E[Z]| \geq t\right] \leq 2\exp\left(-\frac{t^2}{2\sum_{k=1}^n{c_k^2}}\right).
\end{eqnarray*}
and solve for $t$ in terms of $\delta$.
If $|Z-\E[Z]| \geq t$ holds with probability at most~$\delta$, then the complementary event 
$|Z-\E[Z]| \leq t$ holds with probability at least~$1-\delta$.
\end{proof}

Thus, if each summand $Z_k$ is close to its mean $\E[Z_k]$,  then 
with high probability, the sum $Z$ is also close to its mean $\E[Z]$.

In the probabilistic perturbation bound below,
the perturbations $\delta_k$ and $\theta_k$ are represented as independent, bounded, zero-mean random variables.

\begin{theorem}\label{t_1}
Let the perturbations $\delta_k,\theta_k$ be independent random variables with 
$\E[\delta_k]=\E[\theta_k]=0$ and $|\delta_k|,|\theta_k|\leq u$, $1\leq k\leq n$.

Then for any $0<\delta<1$, with probability at least $1-\delta$, the relative forward error in the perturbed inner 
product is bounded by
\begin{eqnarray*}
\left|\frac{\vhx^T\vhy-\vx^T\vy}{\vx^T\vy}\right| 
&\leq& \frac{\|\vx\circ\vy\|_2}{|\vx^T\vy|}\>\sqrt{2\,\ln{(2/\delta)}}\> u(2+u)\\
&=& \frac{\sqrt{\sum_{k=1}^n{|x_ky_k|^2}}}{|\vx^T\vy|}\>\sqrt{2\,\ln{(2/\delta)}}\> u(2+u).
\end{eqnarray*}
\end{theorem}

\begin{proof}
Write the total forward error 
\begin{eqnarray*}
Z\equiv \vhx^T\vhy-\vx^T\vy=Z_1 + \cdots + Z_n
\end{eqnarray*}
as a sum of independent random variables, where each summand represents a "local" forward error,
\begin{eqnarray*}
Z_k\equiv x_ky_k\, \left((1+\delta_k)(1+\theta_k)-1\right)=x_ky_k\,\left(\delta_k+\theta_k+\delta_k\theta_k\right), \qquad 1\leq k\leq n.
\end{eqnarray*}
From the linearity of the mean and $\delta_k$, $\theta_k$ being independent random variables with $\E[\delta_k]=\E[\theta_k]=0$
follows
\begin{eqnarray*}
\E[Z_k]=x_ky_k\>\left(\E[\delta_k] +\E[\theta_k] +\E[\delta_k]\E[\theta_k]\right)= 0, \qquad 1\leq k\leq n.
\end{eqnarray*}
The boundedness of $\delta_k$ and $\theta_k$ implies that the deviation of $Z_k$ from its mean $\E[Z_k]=0$ equals
\begin{eqnarray*}\label{e_H1}
\left|Z_k-\E[Z_jk\right| = |Z_k|= |x_ky_k|\,|\delta_k+\theta_k+\delta_k\theta_k|\leq c_k\equiv |x_ky_k|\,\tau, \qquad 1\leq k\leq n,
\end{eqnarray*}
where $\tau\equiv 2u+u^2=u(2+u)$.
Therefore, the conditions of Lemma~\ref{l_1} are satisfied, and we have
\begin{eqnarray*}
\sum_{k=1}^n{c_k^2}=\sum_{k=1}^n{|x_ky_k|^2}\,\tau^2=\|\vx\circ\vy\|_2^2\,\tau^2.
\end{eqnarray*}
The linearity of the expected value implies
\begin{eqnarray*}
\E[\vhx^T\vhy-\vx^T\vy]=\E[Z]=\E[Z_1]+\cdots+\E[Z_n]= 0.
\end{eqnarray*}
Apply Lemma~\ref{l_1} to conclude that  for any $0<\delta<1$, with probability at least $1-\delta$,
\begin{eqnarray*}
\left|\vhx^T\vhy-\vx^T\vy\right| =\left|Z-\E[Z]\right|
\leq \|\vx\circ\vy\|_2\>\sqrt{2\,\ln{(2/\delta)}}\> \tau.
\end{eqnarray*}
At last divide both sides of the inequality by the constant $|\vx^T\vy|$.
\end{proof}

\begin{remark}[Comparsion]\label{r_1}
The probabilistic bound in Theorem~\ref{t_1} is by a factor of $\sqrt{n}$ tighter than the deterministic
bound (\ref{e_pb2}) in Corollary~\ref{c_2}.

The probabilistic bound in Theorem~\ref{t_1} holds with probability at least $1-\delta$,
\begin{eqnarray*}
\left|\frac{\vhx^T\vhy-\vx^T\vy}{\vx^T\vy}\right| \leq \frac{\|\vx\circ\vy\|_2}{|\vx^T\vy|}
\>\sqrt{2\,\ln{(2/\delta)}}\> u(2+u),
\end{eqnarray*}
while the deterministic bound~(\ref{e_pb2}) equals
\begin{eqnarray*}
\left|\frac{\vhx^T\vhy-\vx^T\vy}{\vx^T\vy}\right| \leq \frac{\|\vx\circ\vy\|_2}{|\vx^T\vy|}
\>\sqrt{n}\>u(2+u).
\end{eqnarray*}
The two bounds differ in the factors $\sqrt{2\,\ln{(2/\delta)}}$ versus $\sqrt{n}$, which implies:
\begin{enumerate}
\item The deterministic bound depends explicitly on the dimension~$n$, 
while the probabilistic bound does not.
\item The probabilistic bound is tighter than the deterministic bound for $n>2 \ln{(2/\delta)}$. 
Specifically, with a tiny failure probability of $\delta=10^{-16}$, the probabilistic bound is tighter 
for $n>76$, and $\sqrt{2\,\ln{(2/\delta)}}\leq 9$.
\end{enumerate}
The numerical experiments in Section~\ref{s_perturbcomp} illustrate that the probabilistic bound tends to be
at least two orders of magnitude tighter than the deterministic bound.
\end{remark}

\begin{example}\label{ex_1}
We illustrate the behaviour of the amplifier
\begin{eqnarray*}
\kappa_2\equiv \|\vx\circ\vy\|_2/|\vx^T\vy|
\end{eqnarray*}
in the probabilistic bound in Theorem~\ref{t_1} with three very special cases.

\begin{enumerate}
\item \textit{No cancellation:}\\
If all $x_ky_k$ have the same sign, then $\kappa_2^2=
\tfrac{\sum_{k=1}^n{|x_ky_k|^2}}{\left(\sum_{k=1}^n{|x_ky_k|}\right)^2}\leq 1$, so that
\begin{eqnarray*}
\left|\frac{\vhx^T\vhy-\vx^T\vy}{\vx^T\vy}\right| \leq \sqrt{2\,\ln{(2/\delta)}}\> u(2+u).
\end{eqnarray*}
If also $x_ky_k=w\neq 0$ for $1\leq k\leq n$, then $\kappa_2^2=\tfrac{nw^2}{(nw)^2}=\tfrac{1}{n}$, so that
$\kappa_2$ decreases with increasing dimension~$n$,
\begin{eqnarray*}
\left|\frac{\vhx^T\vhy-\vx^T\vy}{\vx^T\vy}\right| \leq \sqrt{\frac{2\,\ln{(2/\delta)}}{n}}\> u(2+u).
\end{eqnarray*}
\item \textit{Severe cancellation:}\\
If $x_ky_k=(-1)^k w$ for $1\leq k\leq n$, $w\neq 0$, and $n$ is odd, then $\kappa_2^2=\tfrac{nw^2}{w^2}=n$, so that
$\kappa_2$ increases with increasing dimension~$n$,
\begin{eqnarray*}
\left|\frac{\vhx^T\vhy-\vx^T\vy}{\vx^T\vy}\right| \leq \sqrt{n}\>\sqrt{2\,\ln{(2/\delta)}}\,u(2+u).
\end{eqnarray*}
\end{enumerate}
\end{example}
\medskip

\section{Probabilistic roundoff error bound, assuming independence of roundoff}\label{s_roundoff1}
After presenting the model for independent roundoffs (Section~\ref{s_model1}), 
we derive a motivating deterministic bound
(Section~\ref{s_motdet1}), followed by the  probabilistic bound (Section~\ref{s_prob1}).

\subsection{Roundoff error model}\label{s_model1}
We assume that the elements of $\vx$ and $\vy$ are floating point numbers, and can be stored exactly.
The inner product is computed via recursive summation \cite[Section 4.1]{Higham2002}, by accumulating partial sums
sequentially from left to right,
\begin{eqnarray*}
z_1=x_1y_1, \qquad z_{k+1} =\sum_{j=1}^{k+1}{x_jy_j}, \qquad 1\leq k\leq n-1.
\end{eqnarray*}
The roundoff error model in Table~\ref{tab_fpt} corresponds to  \cite[(3.1) and (3.2)]{Higham2002}.

\begin{table}[tbhp]
\caption{Traditional roundoff error model (guard digits, no fused multiply-add)}\label{tab_fpt}
\begin{center}
\begin{tabular}{|| l | l |}\hline
Floating point arithmetic$\qquad\qquad\quad$  & Exact computation \\
\hline\hline
$\hz_1  =  x_1y_1\,(1+\theta_1)$ & $z_1=x_1y_1$         \\
$\hz_{k+1} = \left(\hz_{k}+x_{k+1}y_{k+1}\,(1+\theta_{k+1})\right)(1+\delta_{k+1})$ & $z_{k+1} = z_{k}+x_{k+1}y_{k+1}$ \\
\hline
$\hz_{n}=\fl(\vx^T\vy)$ & $z_{n}=\vx^T\vy$\\
\hline 
\end{tabular}
\end{center}
\end{table}

For $0<u< 1$ and $k\geq 1$,  we use the abbreviation
\begin{eqnarray}\label{e_gamma}
\gamma_k &\equiv& (1+u)^k-1 =ku+ \mathcal{O}(u^2).
\end{eqnarray}
If $ku<1$ then  \cite[Lemma 3.1]{Higham2002}
\begin{eqnarray*}
\gamma_k\leq  \frac{ku}{1-ku}.
\end{eqnarray*}

\subsection{A motivating deterministic bound}\label{s_motdet1}
First we unravel the expressions for the computed partial sums, and then bound the sums in terms of inputs and the roundoffs.

\begin{lemma}
The partial sums in Table~\ref{tab_fpt} are equal to
\begin{eqnarray*}
\hz_1 &= & x_1 y_1\,(1+\theta_1)\\
\hz_k & =& x_1y_1\,(1+\theta_1)\,\prod_{\ell=2}^k{(1+\delta_{\ell})}+ \sum_{j=2}^k{x_jy_j\,(1+\theta_j)\,\prod_{\ell=j}^k{(1+\delta_{\ell})}},
\qquad 2\leq k\leq n.
\end{eqnarray*}
If $|\delta_k|, |\theta_k|\leq u$, $1\leq k\leq n$, then the partial sums are bounded by
\begin{eqnarray*}
|\hz_1| & \leq & |x_1y_1|\,(1+u)\\
|\hz_k| & \leq & |x_1y_1|\, (1+u)^k + \sum_{j=2}^k{|x_jy_j|\,(1+u)^{k-j+2}}, \qquad 2\leq k\leq n.
\end{eqnarray*}
\end{lemma}

\begin{lemma}\label{l_fpt2}
The total forward error for the computed inner product $\hz_n=\fl(\vx^T\vy)$  in Table~\ref{tab_fpt}
is expressed as a sum of "local forward errors", 
\begin{eqnarray*}
\fl(\vx^T\vy)-\vx^T\vy=\hz_n-z_n=Z_1+\cdots + Z_n,
\end{eqnarray*}
with a local forward error for each summand,
\begin{eqnarray*}
Z_1 &\equiv & x_1y_1\,\left((1+\theta_1)\,\prod_{\ell=2}^n{(1+\delta_{\ell})}-1\right)\\
Z_k & \equiv & x_ky_k\,\left((1+\theta_k)\,\prod_{\ell=k}^n{(1+\delta_{\ell})}-1\right),
\qquad 2\leq k\leq n
\end{eqnarray*}
If $|\delta_k|, |\theta_k|\leq u$, $1\leq k\leq n$, and $\gamma_k$ as in (\ref{e_gamma}), then 
\begin{eqnarray*}
|Z_1| &\leq & c_1\equiv |x_1y_1|\,\gamma_n\\
|Z_k| &\leq  & c_k\equiv |x_ky_k|\, \gamma_{n-k+2}, \qquad 2\leq k\leq n.
\end{eqnarray*}
\end{lemma}

\begin{proof} 
This is analogous to \cite[Lemma 3.1]{Higham2002}.
\end{proof}

Now we can bound the total forward error.

\begin{theorem}
Let  the roundoffs satisfy $|\delta_k|, |\theta_k|\leq u$, $1\leq k\leq n$, with $\gamma_k$ as in~(\ref{e_gamma}).

Then the forward error of the computed inner product $\hz_n=\fl(\vx^T\vy)$ in Table~\ref{tab_fpt} is bounded by
\begin{eqnarray*}
\left|\fl(\vx^T\vy)-\vx^T\vy\right| =\left|\hz_{n}-z_{n}\right|
\leq  \sum_{k=1}^{n}{c_k} =|x_1y_1|\,\gamma_n +\sum_{k=2}^n{|x_ky_k|\,\gamma_{n-k+2}}.
\end{eqnarray*}
\end{theorem}

\begin{proof}
Applying the triangle inequality to the total forward error in Lemma~\ref{l_fpt2} gives
\begin{eqnarray*}
\left|\hz_{n}-z_{n}\right|\leq \sum_{k=1}^n{|Z_k|} \leq \sum_{k=1}^n{c_k}.
\end{eqnarray*}
$\quad$
\end{proof}

The first consequence is the traditional forward error bound \cite[Section 3.1]{Higham2002}.

\begin{corollary}[Traditional bound]\label{c_trad}
Let  the roundoffs satisfy $|\delta_k|, |\theta_k|\leq u$, $1\leq k\leq n$, with $\gamma_k$ as in~(\ref{e_gamma}).

Then the relative forward error of the computed inner product $\hz_n=\fl(\vx^T\vy)$ in Table~\ref{tab_fpt} is bounded by
\begin{eqnarray*}
\left|\frac{\fl(\vx^T\vy)-\vx^T\vy}{|\vx^T\vy|}\right| &\leq& \frac{|\vx|^T|\vy|}{|\vx^T\vy|}\,\gamma_n.
\end{eqnarray*}
\end{corollary}

\begin{proof} 
Define the vectors
\begin{eqnarray*}
\vv\equiv \begin{pmatrix} |x_1y_1| & \cdots  &|x_ny_n|\end{pmatrix}^T, \qquad 
\vg \equiv \begin{pmatrix} \gamma_n & \gamma_n & \gamma_{n-1}& \cdots & \gamma_2\end{pmatrix}^T,
\end{eqnarray*}
and apply the H\"{o}lder inequality to
\begin{eqnarray*}
\sum_{k=1}^n{c_k} =\vv^T\vg\leq \|\vv\|_1\,\|\vg\|_{\infty}=\sum_{k=1}^n{|x_ky_k|}\>\gamma_n=|\vx|^T|\vy|\>\gamma_n.
\end{eqnarray*}
\end{proof}

The second consequence is the motivation for the probabilistic bound to follow.

\begin{corollary}[Deterministic version of Theorem~\ref{t_04}]\label{c_04}
Let the roundoffs satisfy $|\delta_k|, |\theta_k|\leq u$, $1\leq k\leq n$, with $\gamma_k$ as in (\ref{e_gamma}).

Then the relative forward error of the computed inner product $\hz_n=\fl(\vx^T\vy)$ in Table~\ref{tab_fpt} is bounded by
\begin{eqnarray*}
\left|\frac{\fl(\vx^T\vy)-\vx^T\vy}{|\vx^T\vy|}\right| &\leq& \frac{\sqrt{\sum_{k=1}^n{c_k^2}}}{|\vx^T\vy|}\>\sqrt{n}
\end{eqnarray*}
where $c_1\equiv |x_1y_1|\,\gamma_n$, and $c_k\equiv |x_ky_k|\, \gamma_{n-k+2}$, $2\leq k\leq n$.
\end{corollary}

\begin{proof}
Define the non-negative vector
$\ \vc\equiv \begin{pmatrix} c_1 & \cdots & c_n\end{pmatrix}^T$
and use the relation between vector norms
\begin{eqnarray*}
\sum_{k=1}^n{c_k}=\|\vc\|_1\leq \|\vc\|_2\>\sqrt{n} =\sqrt{\sum_{k=1}^n{c_k^2}}\>\sqrt{n}.
\end{eqnarray*}
$\quad$
\end{proof}

\subsection{Probabilistic forward error bound}\label{s_prob1}
Since the roundoffs are independent, bounded zero-mean random variables, we can use Azuma's inequality in Lemma~\ref{l_1}.

\begin{theorem}\label{t_04}
Let the roundoffs $\delta_k, \theta_k$ be independent random variables
with $\E[\delta_k]=\E[\theta_k]=0$ and $|\delta_k|, |\theta_k|\leq u$, $1\leq k\leq n$, and let $\gamma_k$ as in (\ref{e_gamma}).

Then for any $0<\delta<1$, with probability at least $1-\delta$, the relative forward error in 
the computed inner product $\hz_n=\fl(\vx^T\vy)$ in Table~\ref{tab_fpt} is bounded by
\begin{eqnarray*}
\left|\frac{\fl(\vx^T\vy)-\vx^T\vy}{\vx^T\vy}\right|=\left|\frac{\hz_{n}-z_{n}}{z_{n}}\right| 
\leq \frac{\sqrt{\sum_{k=1}^{n}{c_k^2}}}{|\vx^T\vy|}\>\sqrt{2\,\ln{(2/\delta)}},
\end{eqnarray*}
where $c_1\equiv |x_1y_1|\,\gamma_n$, and $c_k\equiv |x_ky_k|\, \gamma_{n-k+2}$, $2\leq k\leq n$.
\end{theorem}

\begin{proof} 
Since the roundoffs are independent random variables, so is 
the total forward error in Lemma~\ref{l_fpt2},
\begin{eqnarray*}
Z\equiv Z_1+\cdots + Z_n=\fl(\vx^T\vy)-\vx^T\vy.
\end{eqnarray*}
The random variables
\begin{eqnarray*}
Z_1 &\equiv & x_1y_1\,\left((1+\theta_1)\,\prod_{\ell=2}^n{(1+\delta_{\ell})}-1\right)\\
Z_k & \equiv & x_ky_k\,\left((1+\theta_k)\,\prod_{\ell=k}^n{(1+\delta_{\ell})}-1\right),
\qquad 2\leq k\leq n,
\end{eqnarray*}
represent the local forward errors and have zero mean, $\E[Z_k]=0$. 
By linearity, the total forward error has zero mean as well, 
\begin{eqnarray*}
\E[Z]=\E\left[Z_1+\cdots + Z_n\right]=\E[Z_1]+\cdots +\E[Z_n]=0.
\end{eqnarray*}
The deviations of the local errors from their means are bounded by
\begin{eqnarray*}
|Z_k-\E[Z_k]|=|Z_k|\leq c_k, \qquad 1\leq k\leq n,
\end{eqnarray*}
with $c_k$ as in Lemma~\ref{l_fpt2}. Thus we can apply Lemma~\ref{l_1} to $Z$, and then
divide both sides by the constant $|\vx^T\vy|$.
\end{proof}

\begin{remark}[Comparison]\label{r_3}
The probabilistic bound in Theorem~\ref{t_04} tends to be tighter than the corresponding deterministic
bound in Corollary~\ref{c_04}.

The probabilistic bound in Theorem~\ref{t_04} holds with probability at least $1-\delta$,
\begin{eqnarray*}
\left|\frac{\vhx^T\vhy-\vx^T\vy}{\vx^T\vy}\right| \leq \frac{\sqrt{\sum_{k=1}^n{c_k^2}}}{|\vx^T\vy|} \>\sqrt{2\,\ln{(2/\delta)}},
\end{eqnarray*}
while the deterministic bound in Corollary~\ref{c_04}  equals
\begin{eqnarray*}
\left|\frac{\vhx^T\vhy-\vx^T\vy}{\vx^T\vy}\right| \leq \frac{\sqrt{\sum_{k=1}^n{c_k^2}}}{|\vx^T\vy|} \>\sqrt{n},
\end{eqnarray*}
where $c_1\equiv |x_1y_1|\,\gamma_n$, and $c_k\equiv |x_ky_k|\, \gamma_{n-k+2}$, $2\leq k\leq n$,
with $\gamma_k$ as in (\ref{e_gamma}).

As in Remark~\ref{r_1}, the two bounds differ in the factors $\sqrt{2\,\ln{(2/\delta)}}$ versus $\sqrt{n}$, which implies:
\begin{enumerate}
\item The deterministic bound depends explicitly on the dimension~$n$, 
while the probabilistic bound does not.
\item The probabilistic bound is tighter than the deterministic bound for $n>2 \ln{(2/\delta)}$. 
Specifically, with a tiny failure probability of $\delta=10^{-16}$, the probabilistic bound is tighter 
for $n>76$, and $\sqrt{2\,\ln{(2/\delta)}}\leq 9$.
\end{enumerate}
The numerical experiments in Section~\ref{s_robi1}
illustrate that the probabilistic expression can be as much as two orders of magnitude tighter then 
the deterministic bound, but stops being an upper bound for non-negative vectors of large dimension.  
\end{remark}
\section{General probabilistic roundoff error bound}\label{s_roundoff2}

In contrast to the previous section, we make no assumptions on the independence of roundoffs.
After presenting the roundoff error model (Section~\ref{s_remodel2}), we derive a motivating deterministic bound
(Section~\ref{s_motdet2}), and then present the probabilistic bound (Section~\ref{s_prob2}), followed by two
upper bounds that take a simpler form (Section~\ref{s_uppersimpler}).

\subsection{Roundoff error model}\label{s_remodel2}
As in Section~\ref{s_model1}, we assume that the elements of $\vx$ and $\vy$ are floating point numbers, and 
can be stored exactly. 
Our model in Table~\ref{tab_fp} differs from the traditional model in Table~\ref{tab_fpt} only in the book keeping.
It distinguishes each step that introduces a roundoff, 
and explicitly separates additions~($+$) from multiplications~($*$). There are $n$ multiplications and $n-1$
additions, so  $2n-1$ distinct roundoffs. 

The model in Table~\ref{tab_fp} is designed 
to do without additional intermediate factors like $x_ky_k(1+\delta_{2k-2})$, and is expressed solely in terms of partial sums.
Since we assume a guard digit model without fused multiply-add,
the roundoff for addition can be recorded in a subsequent step. The very first partial sum
incurs no addition, so we allocate the roundoff to the second 
partial sum for easier indexing.

\begin{table}[tbhp]
\caption{Our roundoff error model (guard digits, no fused multiply-add)}\label{tab_fp}
\begin{center}
\begin{tabular}{|c || l | l |}\hline
$\quad$Operation$\quad$ & Floating point arithmetic$\qquad\qquad\quad$  & Exact computation \\
\hline\hline
$*$ & $\hs_1  =  x_1y_1$ & $s_1=x_1y_1$         \\
    & $\hs_2 = \hs_1\,(1+\delta_1 )$ & $s_2=s_1$ \\
$*$ & $\hs_{2k-1} = \hs_{2k-2}+x_{k}y_{k}\,(1+\delta_{2k-2})$ & $s_{2k-1} = s_{2k-2}+x_ky_k$ \\
$+$ & $\hs_{2k} = \hs_{2k-1}\,(1+\delta_{2k-1})$ & $s_{2k} =s _{2k-1}$ \\
\hline
Output & $\hs_{2n}=\fl(\vx^T\vy)$ & $s_{2n}=\vx^T\vy$\\
\hline 
\end{tabular}
\end{center}
\end{table}

\subsection{A motivating deterministic bound}\label{s_motdet2}
First we bound the computed partial sums in terms of the inputs, and the unit roundoff~$u$.

\begin{lemma}\label{l_shat}
Let the roundoffs satisfy $|\delta_k|\leq u$,  $1\leq k\leq 2n-1$. 

Then the partial sums computed in Table~\ref{tab_fp} are bounded by
\begin{eqnarray*}
|\hs_{2k-1}|&\leq &  |x_1y_1|\>(1+u)^{k-1} +|x_2y_2|\>(1+u)^{k-1}+\cdots  +|x_{k}y_{k}|\> (1+u)\\
&=& |x_1y_1|\>(1+u)^{k-1}+\sum_{j=2}^k{|x_jy_j|\>(1+u)^{k-j+1}}, \qquad 1\leq k\leq n,
\end{eqnarray*}
and
\begin{eqnarray*}
|\hs_{2k}| &\leq &  |x_1y_1|\,(1+u)^k + |x_2y_2|\>(1+u)^k+\cdots + |x_ky_k|\,(1+u)^2\\
&=& |x_1y_1|\>(1+u)^k+ \sum_{j=2}^{k}{|x_jy_j|\>(1+u)^{k-j+2}}, \qquad 1\leq k\leq n.
\end{eqnarray*}
\end{lemma}

\begin{proof} The proof is by induction, starting with the basis for $k=1$,
\begin{eqnarray*}
|\hs_1|&=& |x_1y_1|=|x_1y_1|\>(1+u)^0,\\
|\hs_2| &=& |\hs_1\>(1+\delta_1)| \leq |x_1y_1|\>(1+u).
\end{eqnarray*}
Assuming, as the hypothesis, that the statement of the lemma is correct, 
the induction step gives for $1\leq k\leq n-1$,
\begin{eqnarray*}
|\hs_{2k+1}| &=& |\hs_{2k} +x_{k+1}y_{k+1}\>(1+\delta_{2k})|\leq |\hs_{2k}| +|x_{k+1}y_{k+1}|\>(1+u)\\
&\leq & |x_1y_1|\,(1+u)^{k} +\sum_{j=2}^k{|x_jy_j|\>(1+u)^{k-j+2}} +|x_{k+1}y_{k+1}|\> (1+u)\\
&=& |x_1y_1|\,(1+u)^{k} +\sum_{j=2}^k{|x_jy_j|\>(1+u)^{(k+1)-j+1}} +|x_{k+1}y_{k+1}|\> (1+u)\\
&=& |x_1y_1|\,(1+u)^{k} +\sum_{j=2}^{k+1}{|x_jy_j|\>(1+u)^{(k+1)-j+1}},
\end{eqnarray*}
and for $1\leq k\leq n-1$,
\begin{eqnarray*}
|\hs_{2k+2}|&=& |\hs_{2k+1}\,(1+\delta_{2k+1})|\leq |\hs_{2k+1}|\,(1+u)\\
&=& |x_1y_1|\,(1+u)^{k+1} +\sum_{j=2}^{k+1}{|x_jy_j|\>(1+u)^{(k+1)-j+2}}.
\end{eqnarray*}
$\quad$
\end{proof}

The \textit{total} forward error is 
\begin{eqnarray}\label{e_tfe2}
Z_{2n} \equiv \hs_{2n}-s_{2n} =\fl(\vx^T\vy)-\vx^T\vy,
\end{eqnarray}
while the \textit{partial sum forward errors} are
\begin{eqnarray*}
Z_k \equiv \hs_k-s_k, \qquad 1\leq k\leq 2n, 
\end{eqnarray*}
where $Z_1= 0$.
We use these partial sum errors to distinguish the newly arrived roundoff from the previous roundoffs. 
Then we establish a recursion for the partial sum errors $Z_k$, and bound the difference
between two successive partial sum errors $Z_k$ and $Z_{k-1}$ by the  "incremental error" $c_k\>u$.
This incremental error $c_k\>u$ captures the most recent roundoff introduced when moving from $Z_{k-1}$ to $Z_k$.

\begin{lemma}\label{l_lfe}
The forward errors for the partial sums in Table~\ref{tab_fp} satisfy the recursions
\begin{eqnarray*}
Z_{2k} &=& Z_{2k-1} +\hs_{2k-1}\> \delta_{2k-1}, \qquad 1\leq k\leq n,\\
Z_{2k-1} &=& Z_{2k-2} +x_ky_k\> \delta_{2k-2}, \qquad 2\leq k\leq n.
\end{eqnarray*}
If $|\delta_k|\leq u$, $1\leq k\leq 2n-1$, then 
\begin{eqnarray*}
|Z_{2k}-Z_{2k-1}| \leq  c_{2k-1}\>u, \qquad 1\leq k\leq n,
\end{eqnarray*}
where 
\begin{eqnarray*}
|\hs_{2k-1}|\leq c_{2k-1}\equiv |x_1y_1|\>(1+u)^{k-1}+\sum_{j=2}^{k}{|x_jy_j|\>(1+u)^{k-j+1}},
\end{eqnarray*}
and for $2\leq k\leq n$, 
\begin{eqnarray*}
|Z_{2k-1}-Z_{2k-2}| \leq c_{2k-2}\>u, \qquad \text{where}\quad c_{2k-2} \equiv |x_{k}y_{k}|.
\end{eqnarray*}
\end{lemma}

\begin{proof}
The proof is by induction, following the recursions in Table~\ref{tab_fp}. Since $Z_1=0$,
the induction starts one step later than the one in Lemma~\ref{l_lfe}, and the induction basis is
\begin{eqnarray*}
Z_2 & =& \hs_2 -s_2 = \hs_1(1+\delta_1) -s_1 = Z_1 + \hs_1\>\delta_1,\\
Z_3 &=& \hs_3-s_3 = \hs_2+x_2y_2\>(1+\delta_2)-\left(s_2+x_2y_2\right) =Z_2+x_2y_2\>\delta_2.
\end{eqnarray*}
Assuming, as the hypothesis, that the statement of the lemma is correct, 
the induction step gives for $1\leq k\leq n-1$,
\begin{eqnarray*}
Z_{2k+2}&=&\hs_{2k+2} -s_{2k+2} = \hs_{2k+1}\>(1+\delta_{2k+1}) -s_{2k+1} 
= Z_{2k+1} + \hs_{2k+1}\>\delta_{2k+1}, 
\end{eqnarray*}
and for  $2\leq k\leq n-1$,
\begin{eqnarray*}
Z_{2k+1} &= &\hs_{2k+1}-s_{2k+1} =
\hs_{2k}+x_{k+1}y_{k+1}\>(1+\delta_{2k}) -(s_{2k}+x_{k+1}y_{k+1})\\
& = & Z_{2k} + x_{k+1}y_{k+1}\>\delta_{2k}.
\end{eqnarray*}
Lemma~\ref{l_shat} and the above recursions imply the bounds
\begin{eqnarray*}
|Z_2-Z_1|& =&  |\hs_1\>\delta_1|\leq |\hs_1|\>u\leq c_1\>u \qquad \text{where} \quad c_1=|x_1y_1|,\\
|Z_3-Z_2|& =& |x_2y_2\>\delta_2| \leq |x_2y_2|\>u \qquad \text{where}\quad c_2=|x_2y_2|.
\end{eqnarray*}
In general,
\begin{eqnarray*}
|Z_{2k}-Z_{2k-1}| &=& |\hs_{2k-1}\>\delta_{2k-1}| \leq |\hs_{2k-1}|\> u\leq c_{2k-1}\> u,
\qquad 2\leq k\leq n,
\end{eqnarray*}
where $c_{2k-1}=|x_1y_1|\>(1+u)^{k-1}+\sum_{j=2}^{k}{|x_jy_j|\>(1+u)^{k-j+1}}$, and
\begin{eqnarray*}
|Z_{2k+1}-Z_{2k}| &=& |x_{k+1}y_{k+1}\>\delta_{2k}|\leq |x_{k+1}y_{k+1}|\>u\leq c_{2k}\>u,
\qquad 2\leq k\leq n-1,
\end{eqnarray*}
where $c_{2k}=|x_{k+1}y_{k+1}|$.
\end{proof}

\begin{theorem}[Deterministic version of Theorem~\ref{t_4}]\label{t_det4}
Let the roundoffs satisfy $|\delta_k|\leq u$, $1\leq k\leq 2n-1$.

Then the relative forward error of the computed inner product $\hs_{2n}=\fl(\vx^T\vy)$ in Table~\ref{tab_fp} is bounded by
\begin{eqnarray*}
\left|\frac{\fl(\vx^T\vy)-\vx^T\vy}{\vx^T\vy}\right| =\left|\frac{\hs_{2n}-s_{2n}}{s_{2n}}\right|
\leq  \sqrt{2n-1}\>\frac{\sqrt{\sum_{k=1}^{2n-1}{c_k^2}}}{|\vx^T\vy|}\> u,
\end{eqnarray*}
where 
\begin{eqnarray*}
c_{2k-1} &=& |x_1y_1|\>(1+u)^{k-1}+\sum_{j=2}^{k}{|x_jy_j|(1+u)^{k-j+1}}, \qquad 1\leq k\leq n\\
c_{2k-2} &= &|x_{k}y_{k}| ,\qquad 2\leq k\leq n.
\end{eqnarray*}
\end{theorem}

\begin{proof}
Represent the total error (\ref{e_tfe2}) as a telescoping sum of incremental errors
\begin{eqnarray*}
\fl(\vx^T\vy)-\vx^T\vy=Z_{2n}=(Z_{2n}-Z_{2n-1})+(Z_{2n-1}-Z_{2n-2})+ \cdots +(Z_2-Z_1),
\end{eqnarray*}
where $Z_1=0$.
With the expressions for $c_k$ from Lemma~\ref{l_lfe},
\begin{eqnarray*}
|Z_{2n}| &\leq & \underbrace{|Z_{2n}-Z_{2n-1}|}_{\leq \,c_{2n-1} \,u} +
\underbrace{|Z_{2n-1}-Z_{2n-2}|}_{\leq \,c_{2n-2}\, u}+ \cdots + \underbrace{|Z_2-Z_1|}_{\leq\, c_1\, u}
\leq \sum_{k=1}^{2n-1}{c_k}\, u.
\end{eqnarray*}
As in the proof of Corollary~\ref{c_04}, the relation between the vector one- and two-norms implies
\begin{eqnarray*}
\sum_{k=1}^{2n-1}{c_k}  \leq  \sqrt{2n-1} \> \sqrt{\sum_{k=1}^{2n-1}{c_k^2}}.
\end{eqnarray*}
\end{proof}

\subsection{Probabilistic forward error bound}\label{s_prob2}
We derive a probabilistic bound based on an Azuma Martingale, which does not require independence of roundoffs,
and then compare the probabilistic and deterministic bounds.

\begin{definition}[Martingale, Definition 12.1 in \cite{MitzUp2005}]\label{d_4}
A sequence of random variables $Z_1, Z_2\ldots$ is a Martingale with respect to a sequence 
$\delta_1, \delta_2, \ldots $ if for $k\geq 1$
\begin{enumerate}
\item $Z_k$ is a function of $\delta_1, \ldots, \delta_{k-1}$,
\item $\E[|Z_k|]<\infty$,
\item $\E\left[Z_{k+1} |\delta_1, \ldots, \delta_{k-1}\right]  = Z_k$.
\end{enumerate}
\end{definition}
\medskip

The version of the Martingale below is tailored to our context.
 
\begin{lemma}[Azuma-Hoeffding Martingale, Theorem 12.4 in \cite{MitzUp2005}]\label{l_4}
Let $Z_1,  \ldots, Z_{2n}$ be a Martingale with 
\begin{eqnarray*}
|Z_{k}-Z_{k-1}|\leq c_{k-1}, \qquad 2\leq k\leq 2n.
\end{eqnarray*}
Then for  any $0<\delta<1$ with probability at least $1-\delta$,
\begin{eqnarray*}
\left|Z_{2n}-Z_1\right|\leq \sqrt{\sum_{k=1}^{2n-1}{c_k^2}}\>\sqrt{2\ln{(2/\delta)}}.
\end{eqnarray*}
\end{lemma} 

\begin{proof}
In \cite[Theorem 12.4]{MitzUp2005}, set
\begin{eqnarray*}
\delta \equiv \Pr\left[ |Z_m-Z_0| \geq t\right] \leq 2\exp\left(-\frac{t^2}{2\sum_{k=1}^m{c_k^2}}\right).
\end{eqnarray*}
and $m=2n-1$, and then solve for $t$ in terms of $\delta$.
If $|Z-\E[Z]| \geq t$ holds with probability at most~$\delta$, then the complementary event 
$|Z-\E[Z]| \leq t$ holds with probability at least~$1-\delta$.
\end{proof}

Again, the roundoffs are represented as bounded, zero-mean random variables, but now they are not required
to be independent.
The following bound resembles the one in Theorem~\ref{t_2}, but contains more summands.

\begin{theorem}\label{t_4}
Let the roundoffs $\delta_k$ be random variables with $\E[\delta_k]=0$ and $|\delta_k|\leq u$, $1\leq k\leq 2n-1$.

Then for any $0<\delta<1$, with probability at least $1-\delta$,
the relative forward error of the computed inner product $\hs_{2n}=\fl(\vx^T\vy)$ in Table~\ref{tab_fp} is bounded by
\begin{eqnarray*}
\left|\frac{\fl(\vx^T\vy)-\vx^T\vy}{\vx^T\vy}\right|=\left|\frac{\hs_{2n}-s_{2n}}{s_{2n}}\right| 
\leq \frac{\sqrt{\sum_{k=1}^{2n-1}{c_k^2}}}{|\vx^T\vy|}\>\sqrt{2\,\ln{(2/\delta)}} \> u,
\end{eqnarray*}
where 
\begin{eqnarray*}
c_{2k-1} &= & |x_1y_1|\>(1+u)^{k-1}+\sum_{j=2}^{k}{|x_jy_j|\,(1+u)^{k-j+1}}, \qquad 1\leq k\leq n,\\
c_{2k-2} & = & |x_ky_k|, \qquad 2\leq k\leq n.
\end{eqnarray*}
\end{theorem}

\begin{proof}
Since $Z_1=0$, Table~\ref{tab_fp} implies for the total forward error (\ref{e_tfe2}) that
\begin{eqnarray*} 
|\fl(\vx^T\vy)-\vx^T\vy|=|\hs_{2n}-s_{2n}|=|Z_{2n}|=|Z_{2n}-Z_1|.
\end{eqnarray*}
To apply Lemma~\ref{l_4},
we show that the partial sum forward errors $Z_1, Z_2\ldots, Z_{2n}$ form a Martingale with respect to the roundoffs
$\delta_1,\ldots,\delta_{2n-1}$. To this end, we need to check the conditions in Definition~\ref{d_4} and Lemma~\ref{l_4}.
\begin{enumerate}
\item The recursions in Lemma~\ref{l_lfe}  show that $Z_k$ is a function of the roundoffs $\delta_1,\ldots, \delta_{k-1}$,
$2\leq k\leq 2n$.
\item The expectation of $|Z_k|$ is finite because $|Z_k|$ is a finite sum of bounded summands, and the roundoffs have zero mean. 
\item Lemma~\ref{l_lfe} implies $Z_2=Z_1+x_1y_1\>\delta_1$ where $Z_1=0$.
The linearity of expectation and zero-mean property of the roundoffs implies
\begin{eqnarray*}
\E[Z_2]=\E\left[Z_1+x_1y_1\>\delta_1\right]=Z_1+ x_1y_1\>\E[\delta_1]=Z_1.
\end{eqnarray*}
More generally, item~1 implies that $Z_{2k-2}$ depends on $\delta_1,\ldots, \delta_{2k-3}$, $2\leq k\leq n$.
Conditioning on all of these roundoffs removes the randomness and produces a fixed value,
\begin{eqnarray*}
\E[Z_{2k-2}\, | \delta_1,\ldots,\delta_{2k-3}]=Z_{2k-2}, \qquad 2\leq k\leq n,
\end{eqnarray*}
Combine the above with the zero-mean property of the roundoffs 
\begin{eqnarray*}
\E[\delta_{2k-2}\,|\delta_1,\ldots,\delta_{2k-3}]=\E[\delta_{2k-2}]=0
\end{eqnarray*}
and Lemma~\ref{l_lfe} to conclude
\begin{eqnarray*}
\E[Z_{2k-1}\,|\delta_1,\ldots, \delta_{2k-3}]&=&\E[Z_{2k-2}+x_ky_k\>\delta_{2k-2}\,|\delta_1,\ldots, \delta_{2k-3}] \\
&=& Z_{2k-2}+x_ky_k \E[\delta_{2k-2}]=Z_{2k-2}, \qquad 2\leq k\leq n.
\end{eqnarray*}
Now consider the remaining recursions $Z_{2k}=Z_{2k-1} +\hs_{2k-1}\>\delta_{2k-1}$, $1\leq k\leq n$.
Item~1 and Table~\ref{tab_fp} show that $Z_{2k-1}$ and $\hs_{2k-1}$ depend only on the roundoffs $\delta_1,\ldots,\delta_{2k-2}$.
Conditioning $Z_{2k-1}$ and $\hs_{2k-1}$ on all of these roundoffs removes the randomness and produces  fixed values,
\begin{eqnarray*}
\E[Z_{2k-1}\, | \delta_1,\ldots,\delta_{2k-2}] &=& Z_{2k-1}, \qquad 1\leq k\leq n,\\
\E[\hs_{2k-1}\, | \delta_1,\ldots,\delta_{2k-2}]&=& \hs_{2k-1}, \qquad 1\leq k\leq n.
\end{eqnarray*}
Arguing as above shows
\begin{eqnarray*}
\E[Z_{2k}\,|\delta_1,\ldots, \delta_{2k-2}]&=&\E[Z_{2k-1}+\hs_{2k-1}\delta_{2k-1}\,|\delta_1,\ldots, \delta_{2k-2}] \\
&=&Z_{2k-1}+\hs_{2k-1} \E[\delta_{2k-1}]=Z_{2k-1}, \qquad 1\leq k\leq n.
\end{eqnarray*}
Thus,  $Z_1, Z_2, \ldots, Z_{2n}$ form a Martingale with respect to $\delta_1,\ldots, \delta_{2n-1}$.
\item Lemma~\ref{l_lfe} implies 
\begin{eqnarray*}
|Z_{2k}-Z_{2k-1}| \leq  c_{2k-1}\>u, \qquad 1\leq k\leq n,
\end{eqnarray*}
where 
\begin{eqnarray*}
|\hs_{2k-1}|\leq c_{2k-1}\equiv |x_1y_1|\>(1+u)^{k-1}+\sum_{j=2}^{k}{|x_jy_j|\>(1+u)^{k-j+1}},
\end{eqnarray*}
and for $2\leq k\leq n$, 
\begin{eqnarray*}
|Z_{2k-1}-Z_{2k-2}| \leq c_{2k-2}\>u, \qquad \text{where}\quad c_{2k-2} \equiv |x_{k}y_{k}|.
\end{eqnarray*}
\end{enumerate}
Thus, the conditions for Lemma~\ref{l_lfe} are satisfied, and we can use it to bound $|Z_{2n}-Z_1|$
with the above $c_k$ from Lemma~\ref{l_lfe}.
\end{proof}

\begin{remark}[Comparison]\label{r_4}
The probabilistic bound in Theorem~\ref{t_4} tends to be tighter than the deterministic bound in Theorem~\ref{t_det4}.

The probabilistic bound in Theorem~\ref{t_4} holds with probability at least $1-\delta$,
\begin{eqnarray*}
\left|\frac{\fl(\vx^T\vy)-\vx^T\vy}{\vx^T\vy}\right|=\left|\frac{\hs_{2n}-s_{2n}}{s_{2n}}\right| 
\leq \frac{\sqrt{\sum_{k=1}^{2n-1}{c_k^2}}}{|\vx^T\vy|}\>\sqrt{2\,\ln{(2/\delta)}} \> u,
\end{eqnarray*}
while the deterministic bound in Theorem~\ref{t_det4} is
\begin{eqnarray*}
\left|\frac{\fl(\vx^T\vy)-\vx^T\vy}{\vx^T\vy}\right| =\left|\frac{\hs_{2n}-s_{2n}}{s_{2n}}\right|
\leq \frac{\sqrt{\sum_{k=1}^{2n-1}{c_k^2}}}{|\vx^T\vy|}\> \sqrt{2n-1}\>u,
\end{eqnarray*}
where $c_{2k-1}=|x_{k}y_{k}|$ and $c_{2k} = \sum_{j=1}^{k}{|x_jy_j|(1+u)^{k-j+1}}$, $1\leq k\leq n$.

The two bounds differ in the factors $\sqrt{2\,\ln{(2/\delta)}}$ versus $\sqrt{2n-1}$, which implies:
\begin{enumerate}
\item The deterministic bound increases with the dimension~$n$, 
while the probabilistic bound does not.
\item The probabilistic bound is tighter for $n>\ln{(2/\delta)}+\tfrac{1}{2}$. 
Specifically, with a tiny failure probability of $\delta=10^{-16}$, the probabilistic bound is tighter 
for $n\geq 39$, and $\sqrt{2\,\ln{(2/\delta)}}\leq 9$.
\end{enumerate}
\end{remark}

\subsection{Simpler forward error bounds}\label{s_uppersimpler}
We derive two upper bounds for Theorem \ref{t_det4} and~\ref{t_4} that have a simpler form,
and then conform Wilkinson's intuition \cite[Section 1.33]{Wilk94book}.

The first bound is more compact than Theorem~\ref{t_4}, and makes use of abbreviations for the leading 
subvectors of $|\vx|\circ |\vy|$, and vectors containing powers of $1+u$.

\begin{corollary}[Compact upper bound]\label{c_4a}
Define the $k$-vectors
\begin{eqnarray*} 
(\vx\circ\vy)_k \equiv \begin{pmatrix} |x_1y_1| \\ |x_2y_2|\\\vdots \\ |x_ky_k|\end{pmatrix}, \qquad 
\vu_k\equiv \begin{pmatrix}(1+u)^{k-1} \\ (1+u)^{k-1}\\ \vdots \\ 1+u\end{pmatrix}, \qquad 2\leq k\leq n.
\end{eqnarray*}
If $\tfrac{1}{p}+\tfrac{1}{q}=1$, then in Theorems \ref{t_det4} and~\ref{t_4} we have 
\begin{eqnarray*}
\sum_{k=1}^{2n-1}{c_k^2}\leq \|\vx\circ \vy\|_2^2+\sum_{k=2}^{n}{\|(\vx\circ\vy)_k\|_p^2\, \|\vu_k\|_q^2}.
\end{eqnarray*}
\end{corollary}

\begin{proof}
Partition
\begin{eqnarray*}
\sum_{k=1}^{2n-1}{c_k^2}  & = & \sum_{k=1}^n{c_{2k-1}^2}+\sum_{k=2}^{n}{c_{2k-2}^2}
= \sum_{k=2}^n{c_{2k-1}^2}+ c_1^2+\sum_{k=2}^n{c_{2k-2}^2}.
\end{eqnarray*}
From $c_1=|x_1y_1|$ and $c_{2k-2}=|x_ky_k|$, $2\leq k\leq n$,  follows
\begin{eqnarray*}
 c_1^2+\sum_{k=2}^n{c_{2k-2}^2} = \sum_{k=1}^n{|x_ky_k|^2}=\|\vx\circ\vy\|_2^2.
 \end{eqnarray*}
 Thus $\sum_{k=1}^{2n-1}{c_k^2} =  \|\vx\circ \vy\|_2^2+\sum_{k=2}^n{c_{2k-1}^2}$.
In the remaining sum, apply H\"{o}lder's inequality to each summand, 
\begin{eqnarray*}
c_{2k-1}&=& |x_1y_1|\>(1+u)^{k-1}+\sum_{j=2}^{k}{|x_jy_j|\,(1+u)^{k-j+1}}\\
&=&(\vx\circ\vy)_{k}^T\vu_{k} \leq  \|(\vx \circ \vy)_{k}\|_p\,\|\vu_{k}\|_q, \qquad 2\leq k\leq n.
\end{eqnarray*}
$\quad$
\end{proof}

The second bound, below, takes a much simpler form.

\begin{corollary}[Simplest upper bound for Theorem~\ref{t_4}]\label{c_4}
Let the roundoffs $\delta_k$  be random variables with $\E[\delta_k]=0$ and $|\delta_k|\leq u$, $1\leq k\leq 2n$; 
and let $\gamma_k$ as in (\ref{e_gamma}).

Then for any $0<\delta<1$, with probability at least $1-\delta$,
the relative forward error of the computed inner product $\hs_{2n}=\fl(\vx^T\vy)$ in Table~\ref{tab_fp} is bounded by
\begin{eqnarray}\label{e_re1}
\left|\frac{\fl(\vx^T\vy)-\vx^T\vy}{\vx^T\vy}\right|  \leq  \frac{|\vx|^T|\vy|}{|\vx^T\vy|}\>\sqrt{2\ln{(2/\delta)}}
\>\sqrt{\frac{u\>\gamma_{2n}}{2}}.
\end{eqnarray}
\end{corollary}

\begin{proof}
In Corollary~\ref{c_4a}, choose $p=1$ and $q=\infty$, so that 
\begin{eqnarray*}
\|(\vx\circ\vy)_k\|_1\|\vu_k\|_{\infty} \leq  \|\vx\circ\vy\|_1 (1+u)^{k-1}, \qquad 2\leq k\leq n.
\end{eqnarray*}
The relation between vector norms implies $\|\vx\circ \vy\|_2\leq \|\vx\circ\vy\|_1$.
Insert the preceding two inequalities into Corollary~\ref{c_4a},
\begin{eqnarray*}
\sum_{k=1}^{2n-1}{c_k^2}\leq \|\vx\circ\vy\|_2^2+\sum_{k=2}^{n}{\|(\vx\circ\vy)_k\|_1^2\,\|\vu_k\|_{\infty}^2}
\leq \|\vx\circ\vy\|_1^2\left(1+\sum_{k=2}^{n}{(1+u)^{2(k-1)}}\right).
\end{eqnarray*}
The second factor is a geometric sum,
\begin{eqnarray*}
1+\sum_{k=1}^{n-1}{(1+u)^{2k}}=\sum_{k=0}^{n-1}{(1+u)^{2k}}=\frac{(1+u)^{2n}-1}{(1+u)^2-1}=\frac{\gamma_{2n}}{u^2+2u}.
\end{eqnarray*}
Combining the preceding inequalities gives
\begin{eqnarray*}
\sqrt{\sum_{k=1}^{2n-1}{c_k^2}}\leq \|\vx\circ\vy\|_1\,\sqrt{\frac{\gamma_{2n}}{u^2+2u}}
 \leq \|\vx\circ\vy\|_1\,\sqrt{\frac{\gamma_{2n}\,u}{2}}.
\end{eqnarray*}
At last substitute this into Theorem~\ref{t_4}.
\end{proof}

\begin{remark}[Comparison with traditional bound]\label{r_2}
We quantify and confirm Wilkinson's intuition \cite[Section 1.33]{Wilk94book}, by illustrating 
that the probabilistic bounds
in Theorem~\ref{t_4}, and Corollaries \ref{c_4a} and~\ref{c_4} are proportional to $\sqrt{n}\>u$,
while the traditional bound in Corollary~\ref{c_trad} is proportional to $n\> u$.

Let $\gamma_k=(1+u)^k-1$, $k\geq 1$, be as in (\ref{e_gamma}). 
The probabilistic bound in Corollary~\ref{c_4} holds with probability at least $1-\delta$,
\begin{eqnarray*}
\left|\frac{\fl(\vx^T\vy)-\vx^T\vy}{\vx^T\vy}\right|  \leq  \frac{|\vx|^T|\vy|}{|\vx^T\vy|}\>\sqrt{2\ln{(2/\delta)}}
\>\sqrt{\frac{u\>\gamma_{2n}}{2}},
\end{eqnarray*}
while the deterministic bound in Corollary~\ref{c_trad} equals
\begin{eqnarray*}
\left|\frac{\fl(\vx^T\vy)-\vx^T\vy}{|\vx^T\vy|}\right| &\leq& \frac{|\vx|^T|\vy|}{|\vx^T\vy|}\,\gamma_n.
\end{eqnarray*}
For large $n$, the bounds behave asymptotically like their first order terms,
\begin{eqnarray*}
\gamma_n\approx n\>u, \qquad \sqrt{\frac{u\>\gamma_{2n}}{2}}\approx \sqrt{n} \> u.
\end{eqnarray*}
For small $n$ with $2n\>u<1$, one can bound \cite[Lemma 3.1]{Higham2002},
\begin{eqnarray*}
\gamma_n\leq \frac{nu}{1-nu}, \qquad \sqrt{\frac{u\>\gamma_{2n}}{2}} \leq \frac{\sqrt{n}\> u}{\sqrt{1-2n\>u}}
\end{eqnarray*}
Thus, the probabilistic bound is proportional to $\sqrt{n}\>u$.

Furthermore, $\gamma_n>\sqrt{u\>\gamma_{2n}/2}$ for $n\geq 2$.
With a failure probability of $\delta=10^{-16}$, the probabilistic bound is tighter than the deterministic bound for $n>80$. 
\end{remark}

\section{Numerical experiments}\label{s_numexp}
After describing the setup for the experiments (Section~\ref{s_setup}), we present experiments for the 
perturbation bounds (Section~\ref{s_eperturb}), the roundoff error bounds assuming independence
(Section~\ref{s_robi1}), and the general roundoff error bounds (Section~\ref{s_rog}). 

\subsection{Experimental Setup}\label{s_setup}
We use a tiny failure probability of $\delta=10^{-16}$, which gives a probabilistic factor of $\sqrt{2\,\ln{(2/\delta)}}\leq 8.7$.

Two types of vectors $\vx$ and $\vy$ of dimension up to $n=10^8$ will be considered:
 \begin{itemize}
 \item The elements of $\vx$ and $\vy$ can have different signs. Specifically, $x_j$ and $y_j$ are 
 iid\footnote{independent identically distributed} standard normal random variables with mean~0 and variance~1, 
and $\vx$ and $\vy$ are generated with the Matlab commands
\begin{quote}
\texttt{x = \texttt{single(rand(n, 1))}, y = \texttt{single(rand(n, 1))}}
\end{quote}
\item The elements of $\vx$ and $\vy$ all have the same sign. Specifically, $x_j$ and $y_j$ are
absolute values of iid standard normal random variables,
and $\vx$ and $\vy$ are generated with the Matlab commands 
\begin{quote}
\texttt{x = \texttt{single(abs(rand(n, 1)))}, y = \texttt{single(abs(rand(n, 1)))}}
\end{quote}
\end{itemize}

The exact inner products $\vx^T\vy$ are represented by the double precision computation \texttt{dot(double(x), double(y))}
with unit roundoff $2^{-53}\approx 1.11\cdot 10^{-16}$. Bounds are computed in double precision.
Computations were performed in Matlab R2017a, on a 3.1GHz Intel Core i7 processor. 

\subsection{Experiments for the perturbation bounds}\label{s_eperturb}
 We illustrate the perturbation bounds in Section~\ref{s_perturb}. Here the vectors $\vx$ and $\vy$ are perturbed, while the
 computations are exact.
 
We select single precision perturbations $\delta_j$ and $\theta_j$ that are uniformly distributed in $[-u, u]$,
where $u=2^{-24}\approx 5.96\cdot 10^{-8}$ is the single precision roundoff, and generate
the perturbation vectors $\vdelta$ and $\vtheta$ each with the Matlab command 
\begin{quote}
\texttt{u * (2 * double(single(rand(n, 1))) - ones(n, 1))}.
\end{quote}
 The inner product of the perturbed vectors $\vhx^T\vhy$ is represented by the double precision computation 
 \texttt{dot(double(xh), double(yh))}.

\subsubsection{Amplifiers in Corollary~\ref{c_2}}\label{s_perturbamp}
We compare the amplifiers of $u(2+u)$ in the upper bounds of Corollary~\ref{c_2}, listed again below,
\begin{eqnarray}\label{e_amp}
\kappa_1&\equiv& \frac{\|\vx\circ\vy\|_1}{|\vx^T\vy|}=\frac{|\vx|^T|\vy|}{|\vx^T\vy|}=\frac{\sum_{j=1}^n{|x_jy_j|}}{|\vx^T\vy|}\\
\kappa_2 &\equiv& \sqrt{n}\>\frac{\|\vx\circ\vy\|_2}{|\vx^T\vy|}=\sqrt{n}\>\frac{\sqrt{\sum_{j=1}^{n}{|x_jy_j|^2}}}{|\vx^T\vy|}\nonumber\\
\kappa_{\infty} &\equiv& n\>\frac{\|\vx\circ\vy\|_{\infty}}{|\vx^T\vy|}=n\>\frac{\max_{1\leq j\leq n}{|x_jy_j|}}{|\vx^T\vy|}.\nonumber
\end{eqnarray}
Figure~\ref{f_amp} illustrates that, among the three amplifiers in (\ref{e_amp}), the traditional $\kappa_1$ tends to be the lowest. 
It also illustrates that
amplification of roundoff can be orders of magnitude larger for vector elements with different signs, compared to vectors where
all elements have the same sign.

\begin{figure}[h]
\begin{center}
\includegraphics[width = 2.5in]{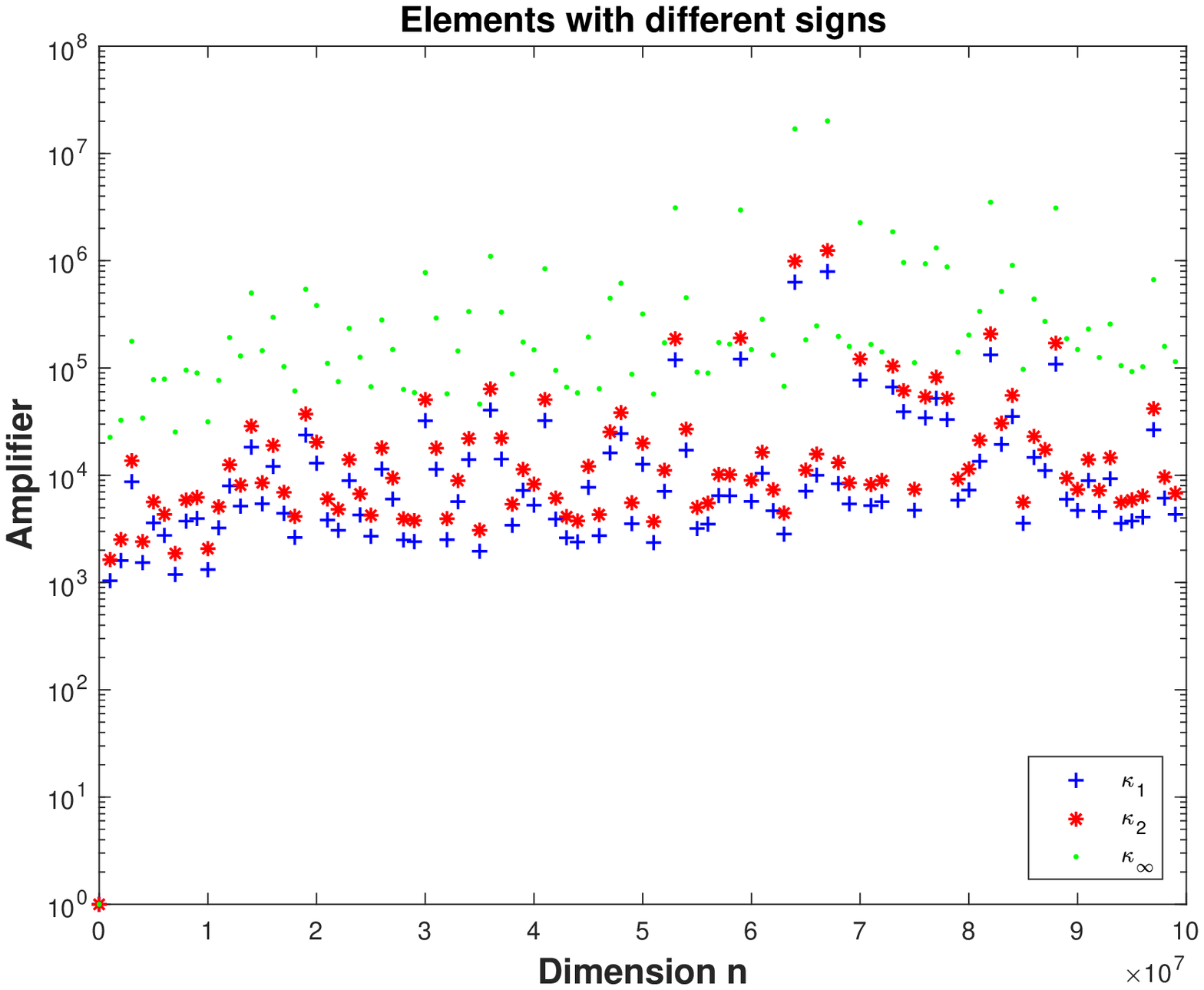}
\includegraphics[width = 2.5in]{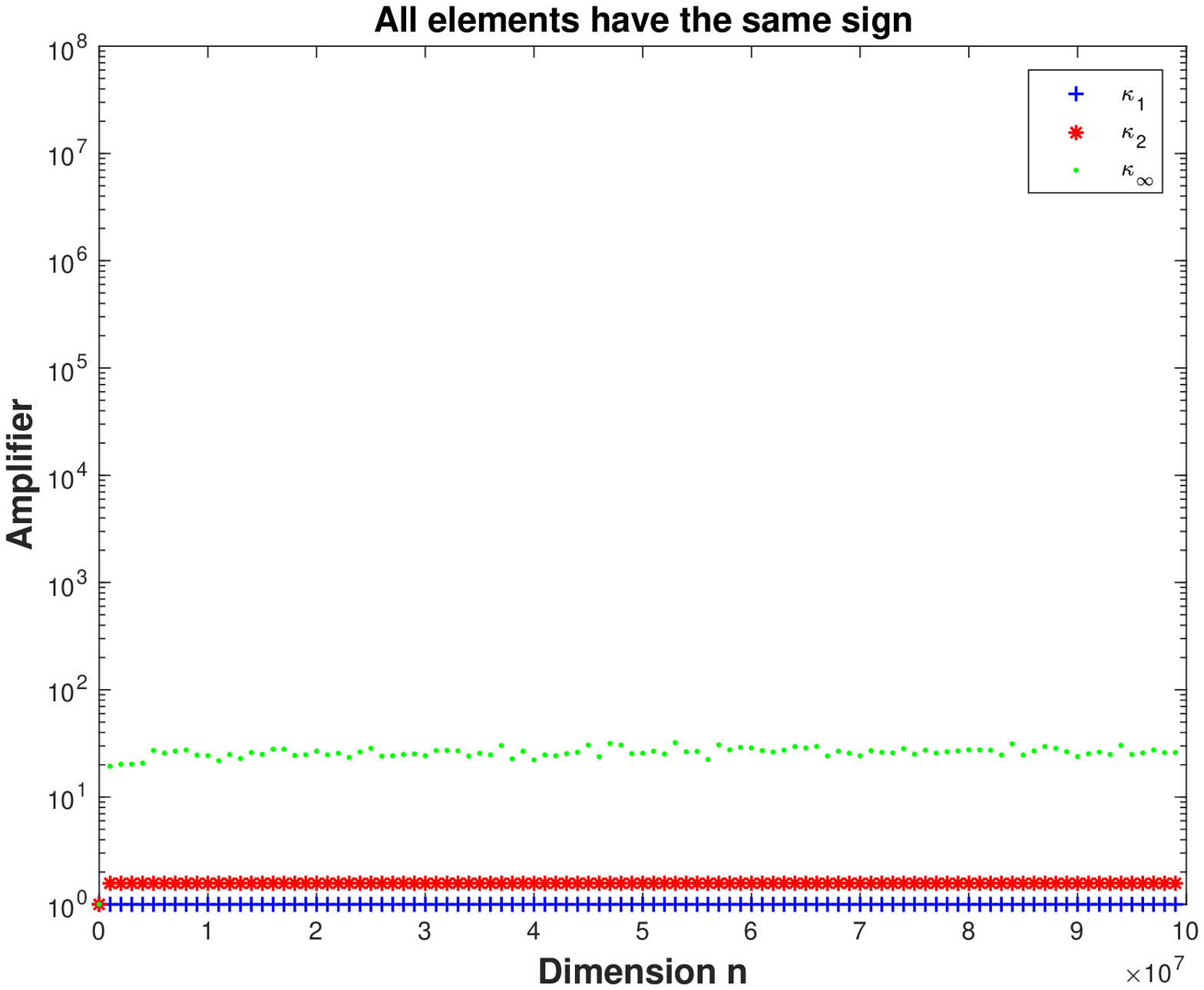}
\end{center}
\caption{Comparison of amplifiers  in (\ref{e_amp}): $\kappa_1$ (blue), $\kappa_2$ (red), and $\kappa_{\infty}$ (green) versus
vector dimensions $1\leq n\leq 10^8$ in steps of $10^6$. Vertical axis starts at 1 and ends at $10^8$.
Left panel: Elements can have different signs. Right panel: All elements have the same sign.}
\label{f_amp}
\end{figure}

\subsubsection{Probabilistic perturbation bound in Theorem~\ref{t_1} and Remark~\ref{r_1}}\label{s_perturbcomp}
This experiment follows up on Remark~\ref{r_1}, where we compare the probabilistic bound from Theorem~\ref{t_1}
to the corresponding deterministic bound from Corollary~\ref{c_2}.
 \begin{itemize}
 \item Deterministic bound
\begin{eqnarray}
\left|\frac{\vhx^T\vhy-\vx^T\vy}{\vx^T\vy}\right| & \leq & \frac{\|\vx\circ\vy\|_2}{|\vx^T\vy|}\>\sqrt{n}\>u(2+u)\label{e_eperturb1}.
\end{eqnarray}
\item Probabilistic bound holding with probability at least $1-\delta$,
\begin{eqnarray}
\left|\frac{\vhx^T\vhy-\vx^T\vy}{\vx^T\vy}\right| &\leq & \frac{\|\vx\circ\vy\|_2}{|\vx^T\vy|}\>\sqrt{2\,\ln{(2/\delta)}}\> u(2+u)\label{e_eperturb2}.
\end{eqnarray}
\end{itemize}
Figure~\ref{f_pert} illustrates that the probabilistic bound~(\ref{e_eperturb2})
tends to be at least two orders orders of magnitude tighter than the deterministic bound~(\ref{e_eperturb1}).

\begin{figure}[h]
\begin{center}
\includegraphics[width = 2.5in]{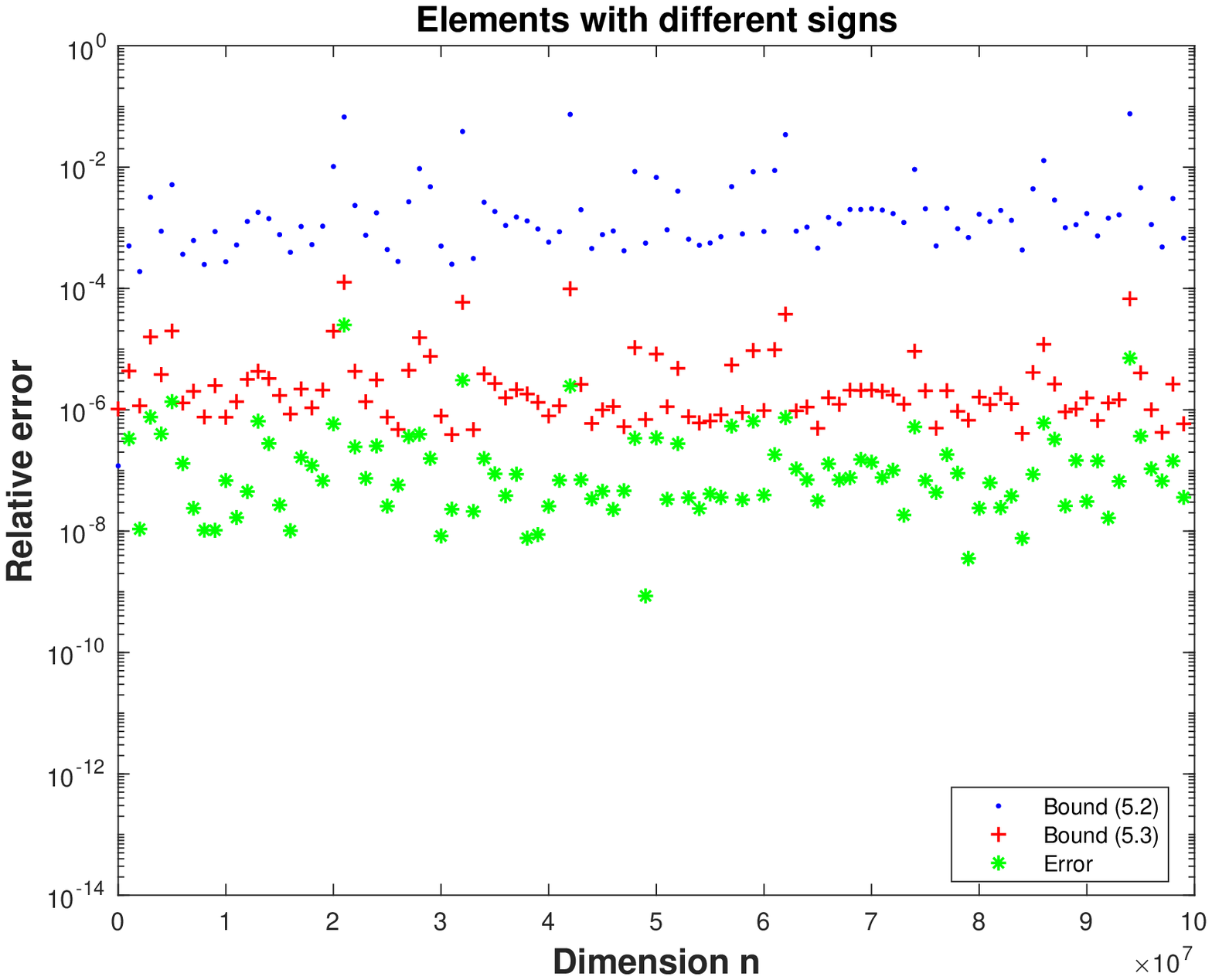}
\includegraphics[width = 2.5in]{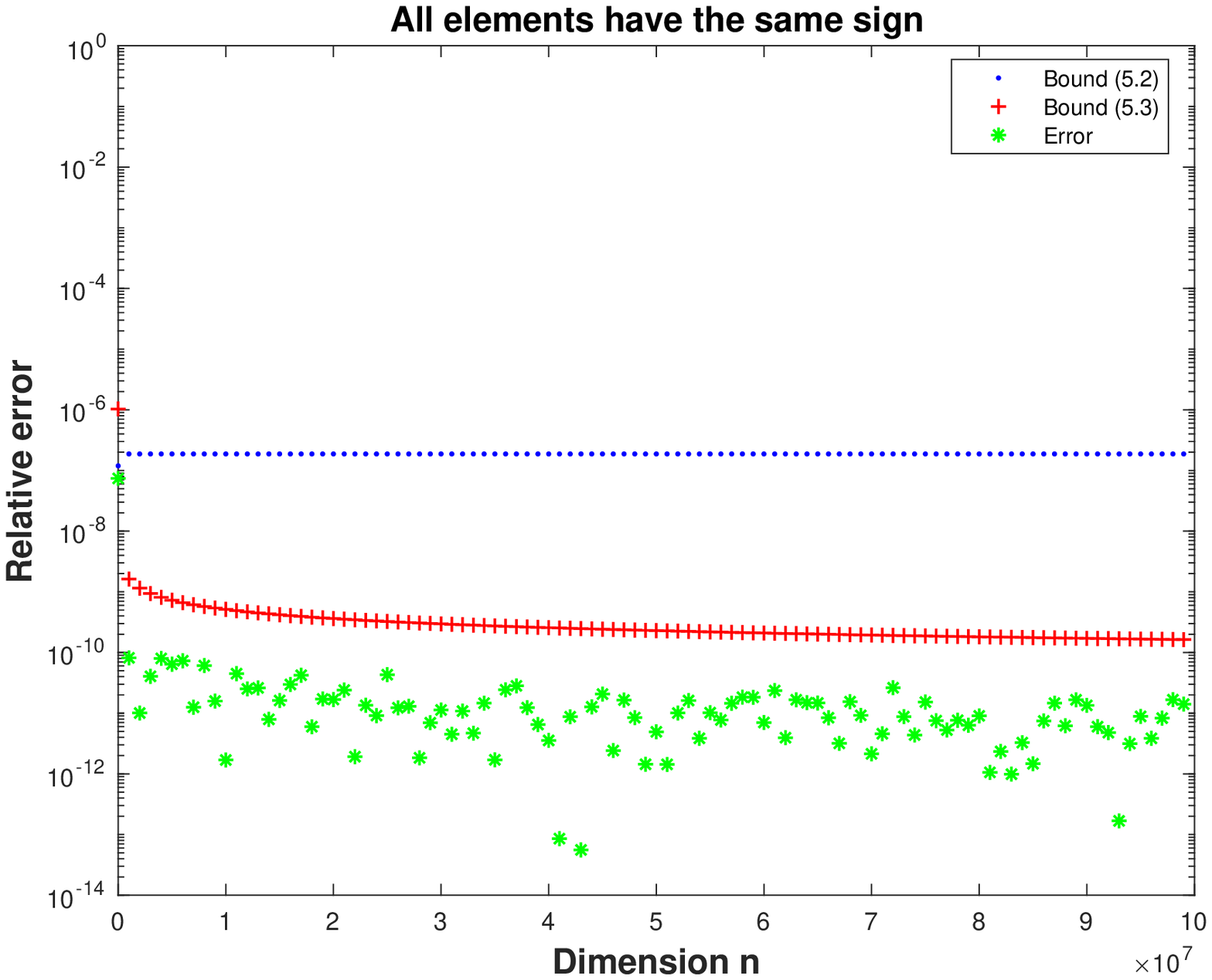}
\end{center}
\caption{Comparison of probabilistic  bound (red \ref{e_eperturb2}) with deterministic bound (blue \ref{e_eperturb1}),
and relative error (green) versus vector dimensions $1\leq n\leq 10^8$ in steps of $10^6$.
Vertical axis starts at $10^{-14}$ and ends at 1.
Left panel: Elements can have different signs. Right panel: All elements have the same sign.}
\label{f_pert}
\end{figure}

\subsection{Experiments for the roundoff error bounds based on independent roundoff}\label{s_robi1}
We illustrate the roundoff error bounds in Section~\ref{s_roundoff1}.

The inner products $\fl(\vx^T\vy)$ are computed in single precision with unit roundoff, in a loop that explicitly stores
the products $x_ky_k$ before adding them to the partial sum, so as to bypass the fused multiply-add.

Specifically, we compare the probabilistic bound in Theorem~\ref{t_04} with the corresponding deterministic bound in Corollary~\ref{c_04}.

 \begin{itemize} 
 \item Deterministic bound
\begin{eqnarray}\label{e_robi1}
\left|\frac{\fl(\vx^T\vy)-\vx^T\vy}{|\vx^T\vy|}\right| &\leq& \frac{\sqrt{\sum_{k=1}^n{c_k^2}}}{|\vx^T\vy|}\>\sqrt{n}
\end{eqnarray}
\item  Probabilistic bound holding with probability at least $1-\delta$,
\begin{eqnarray}\label{e_robi2}
\left|\frac{\fl(\vx^T\vy)-\vx^T\vy}{\vx^T\vy}\right|\leq \frac{\sqrt{\sum_{k=1}^{n}{c_k^2}}}{|\vx^T\vy|}\>\sqrt{2\,\ln{(2/\delta)}},
\end{eqnarray}
\end{itemize}
where $c_1\equiv |x_1y_1|\>\gamma_n$, and $c_k\equiv |x_ky_k|\> \gamma_{n-k+2}$, $2\leq k\leq n$,
and $\gamma_k=(1+u)^k-1$ as in~(\ref{e_gamma}).

Figure~\ref{f_robi1} illustrates that the probabilistic result (\ref{e_robi2}) tends to be two orders of magnitude tighter than the 
deterministic bound (\ref{e_robi1}) for vectors whose elements can have different signs. However, (\ref{e_robi2})
stops being a bound for vectors of large dimension all of whose elements have the same sign.

\begin{figure}[h]
\begin{center}
\includegraphics[width = 2.5in]{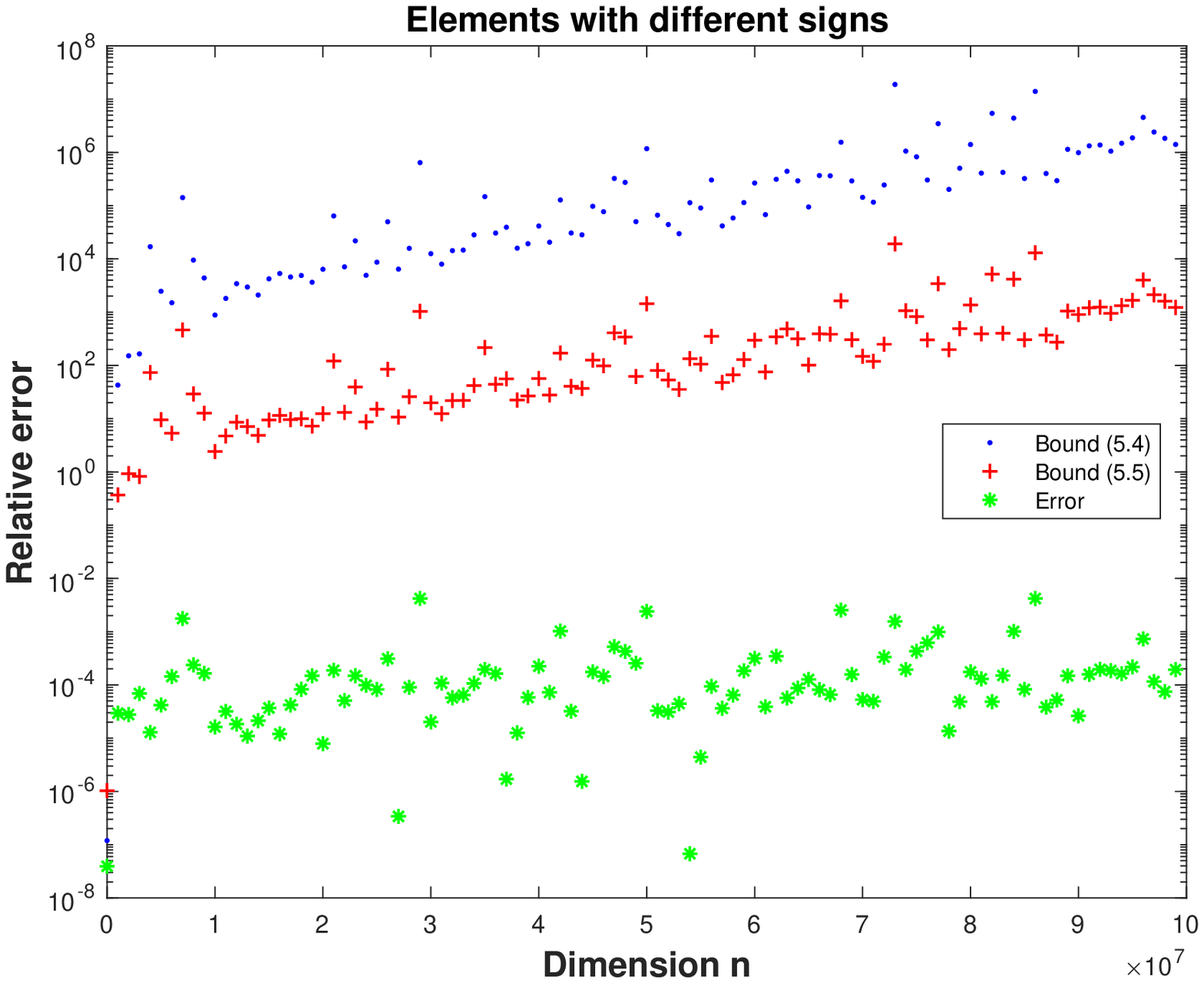}
\includegraphics[width = 2.5in]{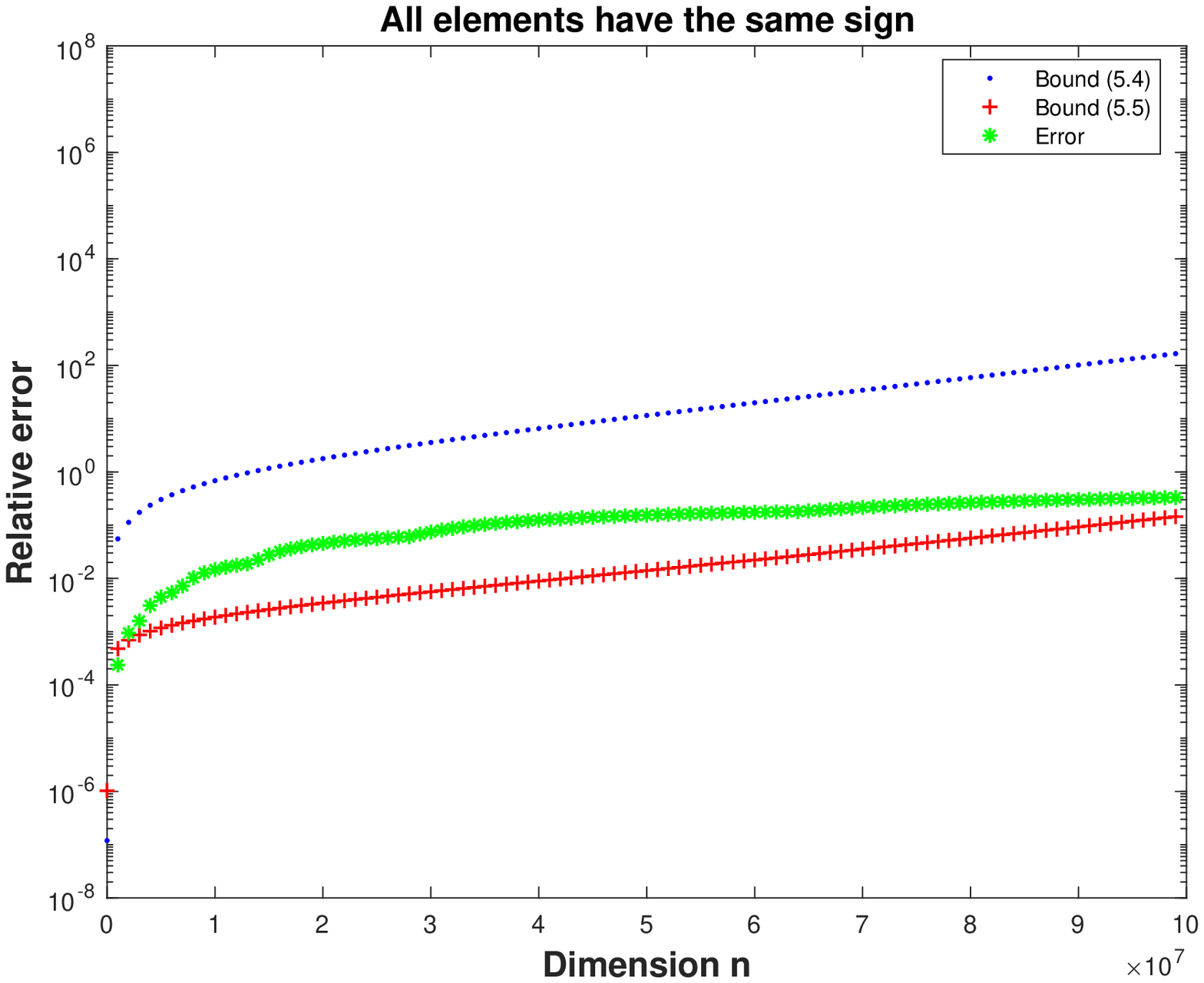}
\end{center}
\caption{Comparison of probabilistic  bound (red \ref{e_robi2}) with deterministic bound (blue \ref{e_robi1}),
and relative error (green) versus vector dimensions $1\leq n\leq 10^8$ in steps of $10^6$.
Vertical axis starts at $10^{-8}$ and ends at $10^8$.
Left panel: Elements can have different signs. Right panel: All elements have the same sign.}
\label{f_robi1}
\end{figure}

Figure~\ref{f_robi1smalllarge} zooms in on the left panel in Figure~\ref{f_robi1} 
and illustrates that (\ref{e_robi2}) remains an upper bound for vector dimensions
up to about $n=10^6$.  The fact that it ceases to be an upper bound for $n>10^6$ does not appear to be a numerical issue,
as nothing changes when the products $|x_ky_k|$ are sorted in increasing or in decreasing order of magnitude.

\begin{figure}[h]
\begin{center}
\includegraphics[width = 2.5in]{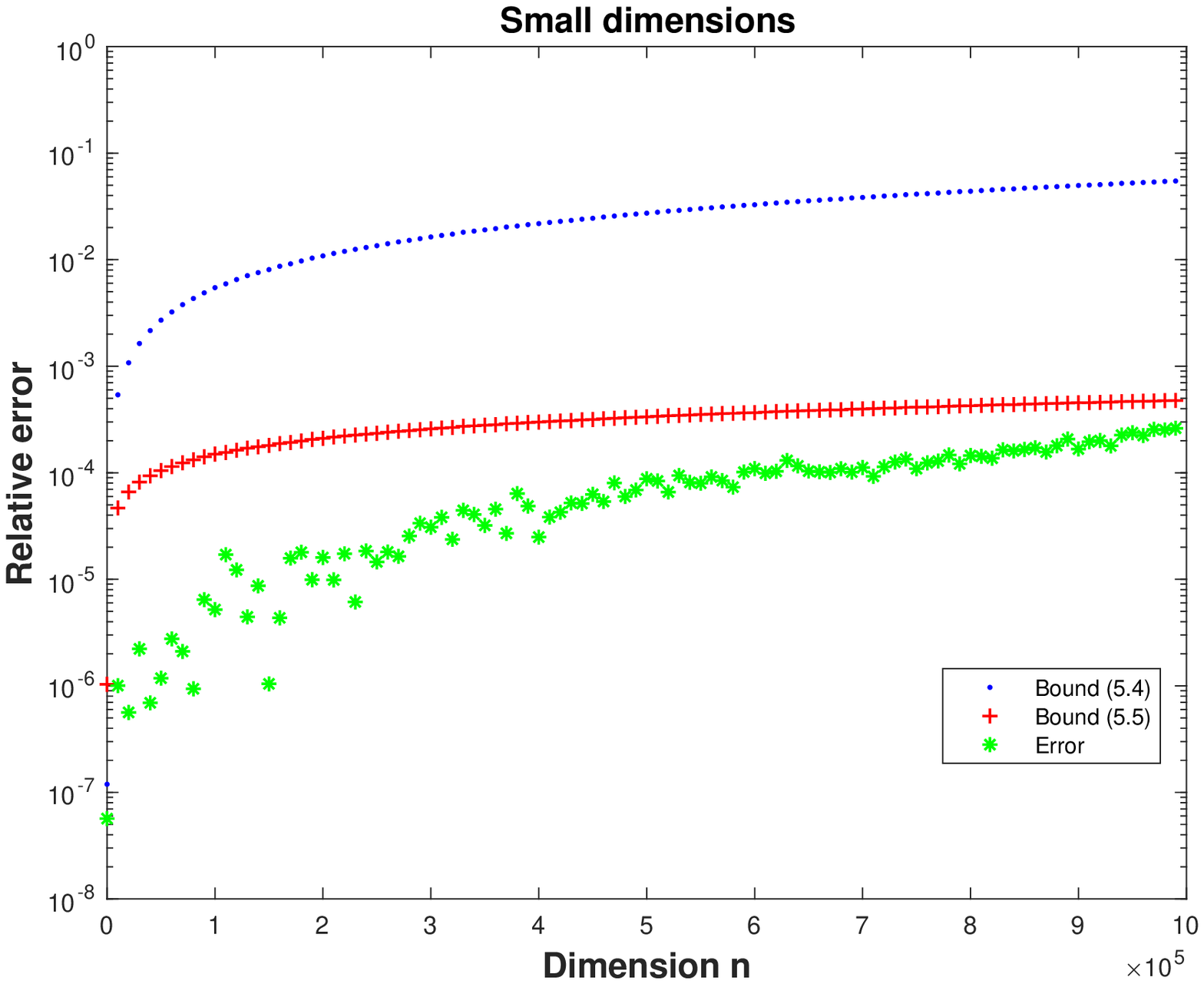}
\includegraphics[width = 2.5in]{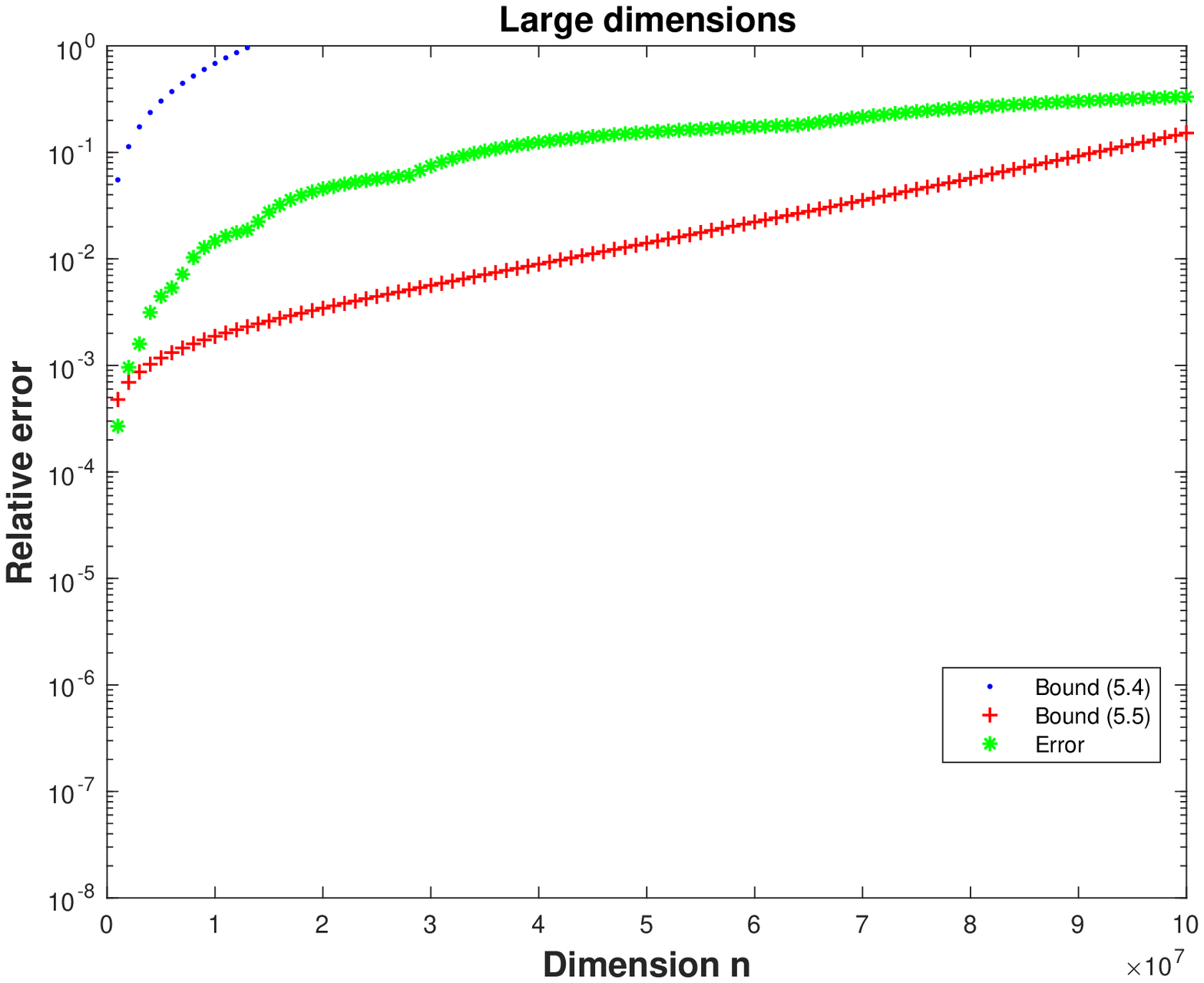}
\end{center}
\caption{Comparison of probabilistic  bound (red \ref{e_robi2}) with deterministic bound (blue \ref{e_robi1}),
and relative error (green) versus vector dimensions when all elements have the same sign.
Vertical axis starts at $10^{-8}$ and ends at $1$.
Left panel: Small dimensions $1\leq n \leq 10^6$ in steps of $10^4$. Right panel: Large dimensions $10^6\leq n\leq 10^8$
in steps of $10^6$.}
\label{f_robi1smalllarge}
\end{figure}

\subsection{Experiments for the general roundoff error bounds}\label{s_rog}
We illustrate the roundoff error bounds in Section~\ref{s_roundoff2}.

As in the previous section, the inner products $\fl(\vx^T\vy)$ are computed in single precision with unit roundoff, in a loop that explicitly stores
the products $x_ky_k$ before adding them to the partial sum, so as to bypass the fused multiply-add.

This experiment follows up on Remark~\ref{r_2}, where we compare the probabilistic bound in Corollary~\ref{c_4}
to the corresponding deterministic bound in Corollary~\ref{c_trad}.
\begin{itemize}
\item Traditional bound 
\begin{eqnarray}\label{e_rog1}
\left|\frac{\fl(\vx^T\vy)-\vx^T\vy}{\vx^T\vy}\right| \leq \frac{|\vx|^T|\vy|}{|\vx^T\vy|}\> \gamma_n,
\end{eqnarray}
\item Probabilistic bound 
\begin{eqnarray}\label{e_rog2}
\left|\frac{\fl(\vx^T\vy)-\vx^T\vy}{\vx^T\vy}\right| \leq 
\frac{|\vx|^T|\vy|}{|\vx^T\vy|}\> \sqrt{\ln{(2/\delta)}}\> \sqrt{\frac{u\>\gamma_{2n}}{2}},
\end{eqnarray}
\end{itemize}
where $\gamma_k=(1+u)^k-1$ as in~(\ref{e_gamma}).

Figure~\ref{f_rog1} illustrates that the probabilistic result (\ref{e_rog2}) tends to be at least two orders of magnitude tighter than the 
deterministic bound (\ref{e_rog1}) for vectors whose elements can have different signs. However, unfortunately,
(\ref{e_rog2}) stops being a bound for vectors of large dimension all of whose elements have the same sign.

\begin{figure}[h]
\begin{center}
\includegraphics[width = 2.5in]{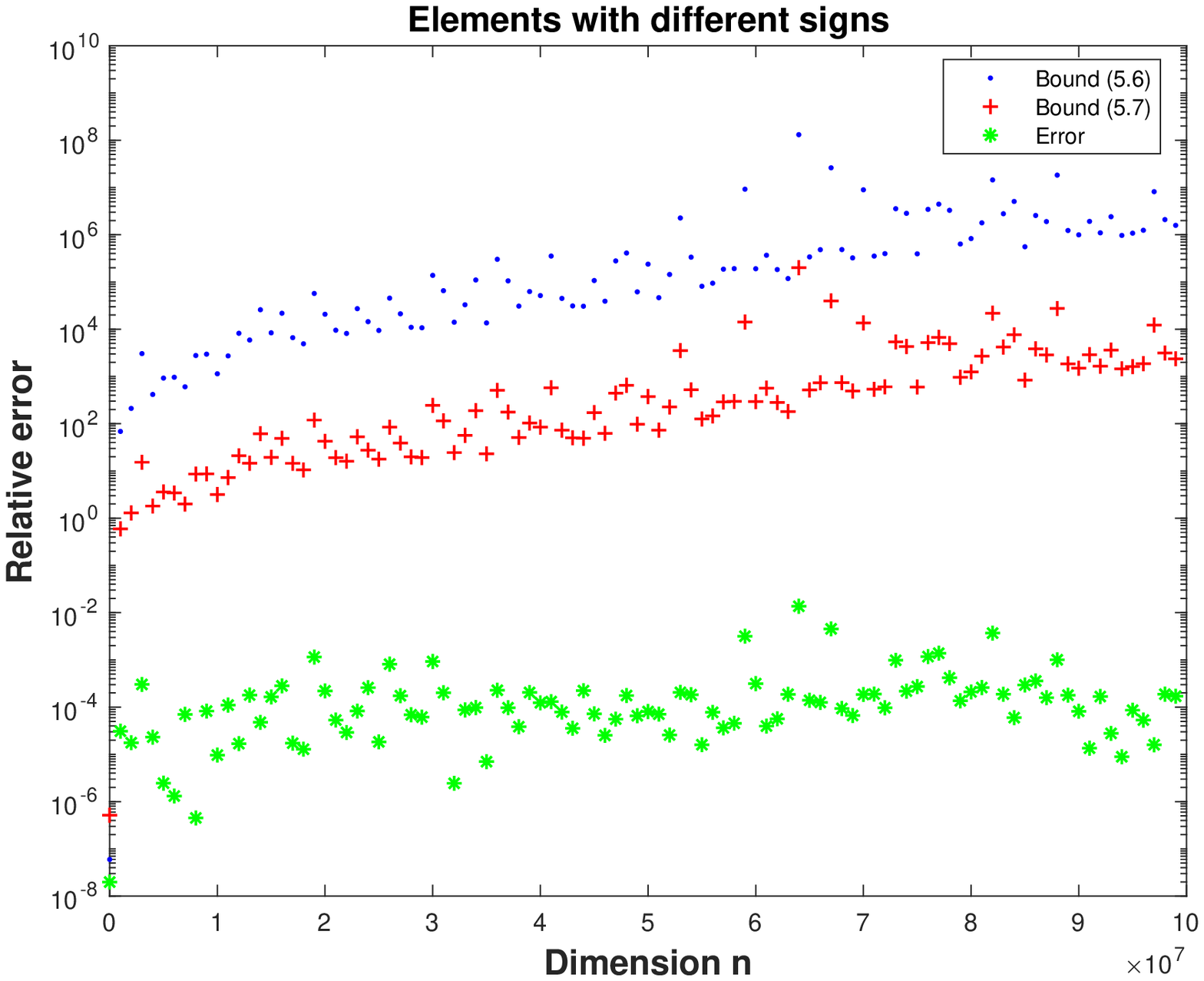}
\includegraphics[width = 2.5in]{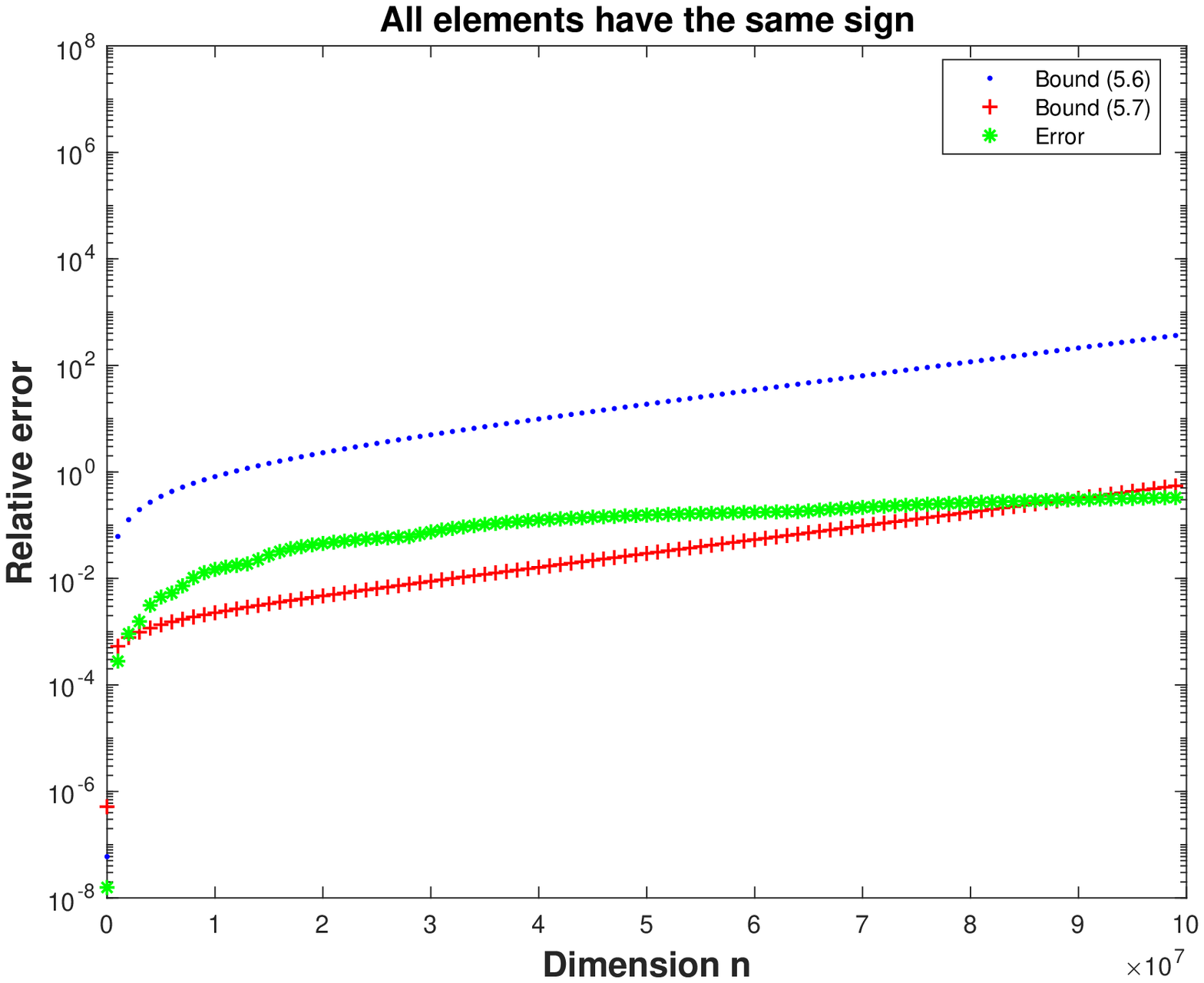}
\end{center}
\caption{Comparison of probabilistic  bound (red \ref{e_rog2}) with deterministic bound (blue \ref{e_rog1}),
and relative error (green) versus vector dimensions $1\leq n\leq 10^8$ in steps of $10^6$.
Vertical axis starts at $10^{-8}$ and ends at $10^8$.
Left panel: Elements can have different signs. Right panel: All elements have the same sign.}
\label{f_rog1}
\end{figure}

\section{Conclusions, and future work}\label{s_conc}
We presented derivations and numerical experiments for probabilistic perturbation and roundoff error bounds for the 
sequentially accumulated inner product of two real $n$-vectors, assuming a guard digit model and no fused multiply-add.
The probabilistic bounds are tighter than the corresponding deterministic bounds, often by several orders of
magnitude.

\paragraph{Issues}
However, for vectors of dimension $n\geq 10^7$ and a tiny failure probability of $\delta=10^{-16}$, 
the probabilistic results are not entirely satisfactory:  On the one hand, they are still too pessimistic for vectors whose
elements have different signs, while on the other hand they stops being upper bounds for vectors all of whose elements 
have the same sign --regardless of whether roundoffs are assumed to be  independent or not.
The latter phenomenon does not appear to be a numerical artifact.

A simple fix would be to adjust the failure probability, making it even more stringent when elements can differ in sign,
while relaxing it when all elements have the same sign.
 However, this does not get to the heart of the problem. Should  the failure probability be explicitly and systematically tied 
 to the dimension~$n$? This would be inconsistent with concentration inequalities, which do not explicitly depend on the
 number of summands. Alternatively, should one not model roundoffs as zero-mean random variables, but instead introduce a bias,
 possibly dimension-dependent,
 for vectors with structure, such as those where all elements have the same sign, see also \cite[section 4.2]{HM18}. 

\section*{Acknowledgements}
We thank Jack Dongarra, Nick Higham, and Clever Moler for helpful discussions. 

\bibliography{LS}

\begin{thebibliography}{10}

\bibitem{BS18}
{\sc I.~Babu\v{s}ka and G.~S\"{o}derlind}, {\em On roundoff error growth in
  elliptic problems}, ACM Trans. Math. Software, 44 (2018), pp.~Art. 33, 22.

\bibitem{BB80}
{\sc E.~H. Bareiss and J.~L. Barlow}, {\em Roundoff error distribution in fixed
  point multiplication}, BIT, 20 (1980), pp.~247--250.

\bibitem{BB85b}
{\sc J.~L. Barlow and E.~H. Bareiss}, {\em On roundoff error distributions in
  floating point and logarithmic arithmetic}, Computing, 34 (1985),
  pp.~325--347.

\bibitem{BB85a}
\leavevmode\vrule height 2pt depth -1.6pt width 23pt, {\em Probabilistic error
  analysis of {G}aussian elimination in floating point and logarithmic
  arithmetic}, Computing, 34 (1985), pp.~349--364.

\bibitem{BBC88}
{\sc M.~Bennani, M.-C. Brunet, and F.~Chatelin}, {\em De l'utilisation en
  calcul matriciel de mod\`eles probabilistes pour la simulation des erreurs de
  calcul}, C. R. Acad. Sci. Paris S\'{e}r. I Math., 307 (1988), pp.~847--850.

\bibitem{BC86}
{\sc M.-C. Brunet and F.~Chatelin}, {\em C{ESTAC}, a tool for a stochastic
  round-off error analysis in scientific computing}, in Numerical mathematics
  and applications ({O}slo, 1985), IMACS Trans. Sci. Comput. 85, I,
  North-Holland, Amsterdam, 1986, pp.~11--20.

\bibitem{Cal91a}
{\sc D.~Calvetti}, {\em Roundoff error for floating point representation of
  real data}, Comm. Statist. Theory Methods, 20 (1991), pp.~2687--2695.

\bibitem{Cal91b}
\leavevmode\vrule height 2pt depth -1.6pt width 23pt, {\em A stochastic
  roundoff error analysis for the fast {F}ourier transform}, Math. Comp., 56
  (1991), pp.~755--774.

\bibitem{Cal92}
\leavevmode\vrule height 2pt depth -1.6pt width 23pt, {\em A stochastic
  roundoff error analysis for the convolution}, Math. Comp., 59 (1992),
  pp.~569--582.

\bibitem{CB90}
{\sc F.~Chatelin and M.-C. Brunet}, {\em A probabilistic round-off error
  propagation model. {A}pplication to the eigenvalue problem}, in Reliable
  numerical computation, Oxford Sci. Publ., Oxford Univ. Press, New York, 1990,
  pp.~139--160.

\bibitem{ChungLu2006}
{\sc F.~Chung and L.~Lu}, {\em Concentration inequalities and {Martingale}
  inequalities: A survey}, Internet Math., 3 (2006), pp.~79--127.

\bibitem{Hen63}
{\sc P.~Henrici}, {\em Problems of stability and error propagation in the
  numerical integration of ordinary differential equations}, in Proc.
  {I}nternat. {C}ongr. {M}athematicians ({S}tockholm 1962), Inst.
  Mittag-Leffler, Djursholm, 1963, pp.~102--113.

\bibitem{Higham2002}
{\sc N.~J. Higham}, {\em Accuracy and Stability of Numerical Algorithms}, SIAM,
  Philadelphia, second~ed., 2002.

\bibitem{HM18}
{\sc N.~J. Higham and T.~Mary}, {\em A new approach to probabilistic rounding
  error analysis}, MIMS EPrint 2018.33, University of Manchester, 2018.

\bibitem{HS66}
{\sc T.~E. Hull and J.~R. Swenson}, {\em Tests of probabilistic models for the
  propagation of roundoff errors}, Comm. ACM, 9 (1966), pp.~108--113.

\bibitem{Kah96}
{\sc W.~Kahan}, {\em The improbability of probabilistic error analyses for
  numerical computations}, March 1996.

\bibitem{MitzUp2005}
{\sc M.~Mitzenmacher and E.~Upfal}, {\em Probability and Computing}, Cambridge
  University Press, Cambridge, 2005.
\newblock Randomized Algorithms and Probabilistic Analysis.

\bibitem{Tie70}
{\sc M.~Tienari}, {\em A statistical model of roundoff error for varying length
  floating-point arithmetic}, Nordisk Tidskr. Informationsbehandling (BIT), 10
  (1970), pp.~355--365.

\bibitem{vNG47}
{\sc J.~{von Neumann} and H.~H. Goldstine}, {\em Numerical inverting of
  matrices of high order}, Bull. Amer. Math. Soc., 53 (1947), pp.~1021--1099.

\bibitem{Wilk94book}
{\sc J.~H. Wilkinson}, {\em Rounding errors in algebraic processes}, Dover
  Publications, Inc., New York, 1994.
\newblock Reprint of the 1963 original [Prentice-Hall, Englewood Cliffs, NJ].

\end{thebibliography}
\end{document}